\documentclass[dvipdfmx]{amsart}

\makeatletter
	
	\@addtoreset{equation}{section}

	\@addtoreset{figure}{section}
\makeatother

\usepackage{amsthm}
\usepackage{amsmath}
\usepackage{amsfonts}
\usepackage{graphicx}
\usepackage{subcaption}
\usepackage{url}
\usepackage{mathtools}
\usepackage{caption}
\usepackage{appendix}
\title[]{On the Annihilating polynomial of the Colored Jones Polynomial for Some Links}
\author[]{Shun Sawabe}
\address{Department of Pure and Applied Mathematics, School of Fundamental Science and Engineering, Waseda University, 3-4-1 Okubo, Shinjuku, Tokyo 169-8555, Japan}
\email{ssawabe[at]aoni.waseda.jp}
\date{}

\newtheorem{thm}{Theorem}[section]
\newtheorem{defi}[thm]{Definition}
\newtheorem{prop}[thm]{Proposition}
\newtheorem{lem}[thm]{Lemma}

\newtheorem{conj}[thm]{Conjecture}
\newtheorem{question}[thm]{Question}

\theoremstyle{definition}
\newtheorem{rem}[thm]{Remark}
\newtheorem{ex}[thm]{Example}

\newcommand{\iu}{\sqrt{-1}}

\newcommand{\A}{\mathcal{A}}
\newcommand{\U}{\mathcal{U}}
\newcommand{\hA}{\widehat{A}}
\newcommand{\hB}{\widehat{B}}
\newcommand{\Af}{\mathfrak{A}}
\newcommand{\cL}{\mathcal{L}}
\newcommand{\hI}{\widehat{I}}

\newcommand{\tQ}{\tilde{Q}}
\newcommand{\tE}{\tilde{E}}

\newcommand{\tc}{\tilde{c}}
\newcommand{\tdel}{\tilde{\delta}}

\DeclareMathOperator{\li}{Li_2}

\begin{document}
\begin{abstract}
	In this paper, we consider polynomials and ideals obtained from the colored Jones polynomial
	in both commutative and noncommutative cases.
	In the commutative case, this ideal contains polynomials that can be regarded as the link version of the $A$-polynomial;
	in the noncommutative case, it consists of annihilating polynomials of the colored Jones polynomial and
	can be regarded as the link version of the $A_q$-polynomial.
	Moreover, we formulate the link version of the AJ conjecture.
\end{abstract}
\maketitle
\section{Introduction} %
\subsection{$A_q$-polynomial and the AJ conjecture}
Since the relationship between the quantum group $\mathcal{U}_q(sl_2)$ and the Jones polynomial was established,
it has been well-known that we can systematically construct knot invariants from quantum groups and their representations.
The colored Jones polynomial, which is a generalization of the Jones polynomial, is one of such invariants.
Let $K$ be a knot, and let $V_K(n)$ be the $n$-th colored Jones polynomial for the knot $K$.
Here, we normalize the colored Jones polynomial so that $V_\bigcirc(n)=[n]$, where
\[
	[n] = \frac{\{n\}}{\{1\}},\quad \{n\}=q^{\frac{n}{2}}-q^{-\frac{n}{2}}
\]
for an integer $n$. Namely, $V_K(n) = [n]J_n(K;q)$.
For any knot $K$, the colored Jones polynomial $V_K(n)$ for the knot $K$ has a nontrivial recurrence relation
\begin{equation} \label{eq:Jrec}
	\sum^d_{j=0}c_j(q,q^n)V_K(n+j) = 0,
\end{equation}
where $c_j(q,q^n) \in \mathbb{Z}[q,q^n]$.
See \cite{GL} for details.
Defining the operators $E$ and $Q$ by
\[
	(EV_K)(n)=V_K(n+1),\quad (QV_K)(n)=q^nV_K(n),
\]
we can rewrite \eqref{eq:Jrec} as
\[
	\left(\sum^d_{j=0}c_j(q,Q)E^j \right)V_K(n) = 0.
\]
This yields a nontrivial annihilating polynomial of $V_K(n)$.
Note that the polynomial is in the noncommutative algebra $\A$ generated by $E$ and $Q$ with relation $EQ=qQE$.
Note also that $V_K(n)$ has infinitely many annihilating polynomials,
and they form an ideal in (the localization of) the algebra $\A$.
The $A_q$-polynomial $\hA(K)(E,Q)$ for a knot $K$
is defined as the generator of the ideal with the smallest $E$-degree in $\A$, and is of the form
\[
	\hA(K)(E,Q) = \sum _k a_k E^k,
\]
where $a_k \in \mathbb{Z}[q,Q]$ are coprime.
The AJ conjecture \cite{Ga} states that the $A_q$-polynomial is a $q$-analogue of the $A$-polynomial $A_K(l,\alpha)$.
First, we recall the definition of the $A$-polynomial.
Let $M$ be the exterior of the knot $K$,
let $R(M)$ be an affine algebraic variety consisting of all $SL(2,\mathbb{C})$-representations of the knot group $ \pi_1(M)$,
and let $X(M)$ be the $SL(2,\mathbb{C})$-character variety of $\pi_1(M)$.
There is a canonical projection $t:R(\partial M) \to X(\partial M)$.
Moreover, the inclusion map $\iota: \partial M \to M$ induces the map $r:X(M) \to X(\partial M)$.
Let $ \Delta $ be a subvariety of $R(\partial M)$ consisting of diagonal representations.
We also let $\lambda$ and $\mu$ be the preferred longitude and the meridian, respectively.
Since the diagonal representation $\varrho$ is determined by the $(1,1)$-entries of $\varrho(\lambda)$ and $\varrho(\mu)$,
we can identify $ \Delta $ with $ (\mathbb{C}^{\times})^2$.
Composing the restriction $t|_{\Delta}:\Delta \to X(\partial M)$ and the identification above,
we obtain the map $\theta : (\mathbb{C}^{\times})^2 \to X(\partial M)$.
\begin{defi}[\cite{CL}]
	Let $D_M$ be the union of $\theta^{-1}(\overline{r(X)})$ with $X$ running over all irreducible components in $X(M)$
	such that $\overline{r(X)}$ is $1$-dimensional.
	The closure of $D_M$ in $\mathbb{C}^2$ is an algebraic curve,
	and we call its $\mathbb{Z}$-coefficient defining polynomial the $A$-polynomial for the knot $K$.
\end{defi}
\begin{conj}[the AJ conjecture \cite{Ga}]
	For any knot $K$, the $A$-polynomial $A_K(l,\alpha)$ for $K$ is equal to $ \varepsilon \hA(K)(l,\alpha^2)$
	up to multiplication by an element in $ \mathbb{Q}(\alpha)$,
	where $ \varepsilon $ is an evaluation map at $q=1$.
\end{conj}
The AJ conjecture is true for large classes of knots such as twist knots \cite{Le}, torus knots \cite{Tr13},
and some classes of hyperbolic knots \cite{LZ}.

\subsection{A link version of the $A$-polynomial}
Let $\cL=\cL_1 \cup \cdots \cup \cL_k$ be a $k$-component link, and let $M$ be its exterior.
The boundary of $M$ is a union of $k$ tori
\[
	\partial M = T_1 \cup \cdots \cup T_k,
\]
where a torus $T_j$ follows the $j$-th component $\cL_j$.
Let $\lambda _j$ and $\mu_j$, with $j=1,\ldots,k$, be the preferred longitude and the meridian of the torus $T_j$.
We also let $\varrho: \pi_1(M) \to SL(2,\mathbb{C})$ be a representation of the link group.
A link version of the $A$-polynomial should similarly express the conditions for the eigenvalues $l_j$ and $\alpha_j$
of $\varrho(\lambda _j)$ and $\varrho(\mu _j)$, with $j=1,\ldots,k$,
that determine the $SL(2,\mathbb{C})$-representation of the link group, just as it does for knots.
In \cite{Mu}, a link version of the $A$-polynomial is calculated from a potential function (see \ref{sec:pf}) of the colored Alexander polynomial.
We can perform a similar calculation via the colored Jones polynomial and obtain the same polynomial.
As we will review later, we can obtain a family of deformations of the hyperbolic structure from saddle points
of the potential function $\Phi(\alpha_1,\ldots,\alpha_k,w_1,\ldots,w_\nu)$
of the $(n_1,\ldots,n_k)$-colored Jones polynomial for the $k$-component link $\cL=\cL_1 \cup \cdots \cup \cL_k$.
Here, the variables $\alpha_1,\ldots,\alpha_k$ derive from the colors $n_1,\ldots,n_k$.
If the equations
\begin{equation} \label{eq:delPhiintro}
	\begin{dcases}
		\exp \left(w_i \frac{\partial \Phi}{\partial w_i} \right)= 1, \quad (i = 1,\ldots,\nu)\\
		\exp \left(\alpha_j \frac{\partial \Phi}{\partial \alpha_j}\right) = l_j^2, \quad (j = 1,\ldots,k) .
	\end{dcases}
\end{equation}
have a common solution, the link complement admits a hyperbolic structure.
In other words, a saddle point provides a representation of the link group.
The equations \eqref{eq:delPhiintro} are equivalent to algebraic equations
\[
	\begin{dcases}
		f_i(\alpha_1,\ldots,\alpha_k,w_1,\ldots,w_\nu)=0, \quad (i = 1,\ldots,\nu)\\
		g_j(l_j,\alpha_1,\ldots,\alpha_k,w_1,\ldots,w_\nu)=0, \quad (j = 1,\ldots,k),
	\end{dcases}
\]
where the left-hand sides are $\mathbb{Z}$-coefficient polynomials in $w_i$'s, $ \alpha _j$'s and $l_j$'s.
Then, as a generator of the ideal
\[
	\langle f_1,\ldots,f_\nu,g_j \rangle \cap \mathbb{Q}(\alpha_1,\ldots,\alpha_k)[l_j] \subset \mathbb{Q}(\alpha_1,\ldots,\alpha_k)[l_j],
\]
we have a polynomial $A_j(\cL)(l_j,\alpha_1,\ldots,\alpha_k)$.
For hyperbolic knots, the polynomial obtained here corresponds to the factor of the $A$-polynomial associated with the hyperbolic structure.
In Section \ref{sec:links}, we therefore call the polynomial the hyperbolic $A$-polynomial for the $j$-th component. 
The polynomials $A_j(\cL)$, with $j=1,\ldots,k$, belong to the ideal
\[
	I_{\cL}^\Phi = \langle f_1,\ldots,f_\nu, g_1,\ldots,g_k \rangle \cap \mathbb{Q}(\alpha_1,\ldots,\alpha_k)[l_1,\ldots,l_k]
\]
in the algebra $\mathbb{Q}(\alpha_1,\ldots,\alpha_k)[l_1,\ldots,l_k]$.
Even for hyperbolic knots, we cannot obtain the factor corresponding to the abelian representation by the above method.

Now let us formulate the ideal $I(\cL)$ in $\mathbb{Q}(\alpha_1,\ldots,\alpha_k)\allowbreak [l_1,\ldots,l_k]$ that expresses
the conditions for defining an arbitrary family of $SL(2,\mathbb{C})$-representations of $\pi_1(M)$ parametrized by $\alpha_1,\ldots, \alpha_k$,
regardless of whether it corresponds to hyperbolic structures or not.
Let $O$ be the set of all possible orientations of the link $\cL$.
We fix an orientation $\sigma \in O$ of the link $\cL$.
Suppose a presentation of the group $\pi_1(M)$ is given by
\begin{equation} \label{eq:grppre}
	\pi_1(M) = \langle x_1,\ldots,x_s \mid r_1,\ldots,r_t \rangle.
\end{equation}
For a representation $\varrho:\pi_1(M) \to SL(2,\mathbb{C})$, we put
\begin{equation} \label{eq:repx}
	\varrho(x_a) = \left(\begin{matrix}
	x_{11}^{(a)} & x_{12}^{(a)} \\
	x_{21}^{(a)} & x_{22}^{(a)}
	\end{matrix}\right),\quad (a=1,\ldots,s).
\end{equation}
Let $\iota_j$ be the inclusion map $\iota_j:T_j \to M$, with $j=1,\ldots,k$.
This induces the map $(\iota_j)_*:\pi_1(T_j) \to \pi_1(M)$. We put $ \varrho_j=\varrho \circ (\iota_j)_*$.
For each $j$, $\varrho_j(\lambda_j)$ commutes with $\varrho_j(\mu_j)$.
Therefore, there exists $g_j=(g^{(j)}_{\alpha \beta}) \in SL(2,\mathbb{C})$ such that
\begin{equation} \label{eq:simdiag}
	g_j^{-1}\varrho_j(\lambda_j)g_j = \left(\begin{matrix}
	l_j & 0 \\
	0 & l_j^{-1}
	\end{matrix}\right),\text{ and } g_j^{-1}\varrho_j(\mu_j)g_j = \left(\begin{matrix}
	\alpha_j & 0 \\
	0 & \alpha_j^{-1}
	\end{matrix}\right).
\end{equation}
We obtain a system of algebraic equations with coefficients in $\mathbb{Q}(\alpha_1,\ldots,\alpha_k)$
as conditions for \eqref{eq:repx} and \eqref{eq:simdiag} to be consistent with the group presentation \eqref{eq:grppre}.
Let $S$ be the set of polynomials that appear in those algebraic equations.
Then, we put
\[
	I^{\sigma}(\cL) = \langle S \rangle \cap \mathbb{Q}(\alpha_1,\ldots,\alpha_k)[l_1,\ldots,l_k].
\]
We define the ideal $I(\cL)$ as the intersection of the ideals $I^{\sigma}(\cL)$ over all $\sigma \in O$:
\[
	I(\cL) = \bigcap_{\sigma \in O}I^{\sigma}(\cL).
\]
See Definition \ref{def:Aideal} for details.
If $\cL$ is hyperbolic, a generator of $I(\cL) \cap \mathbb{Q}(\alpha_1,\ldots,\alpha_k)[l_j]$
would have both factors corresponding to the hyperbolic structure and ones not corresponding to it.
Thus the ideal $I(\cL)$ would be contained in the ideal obtained from only the factors corresponding to the hyperbolic structure.
Specifically, the ideal $I(\cL)$ would not coincide with $I_{\cL}^\Phi$.
\begin{rem}
For the Hopf link $H$ and the Whitehead link $W$,
we obtain a linear polynomial with respect to both $l_1$ and $l_2$
that is in the ideal $I(H)$ or $I_W^\Phi$.
We call such a polynomial a $B$-polynomial.
\end{rem}

\subsection{A link version of the $A_q$-polynomial}
We also define a link version of the $A_q$-polynomial.
In the case of links, for simplicity of calculation, we use the colored Jones polynomial 
$V_{\cL}(n_1,\ldots,n_k) = V_{(n_1,\ldots,n_k)}(\cL;s)$ normalized as
\[
	V_{(n_1,\ldots,n_k)}(\bigcirc^k;s) = [n_1]_s \cdots[n_k]_s,
\]
where $s=q^{\frac{1}{2}}$, $\bigcirc^k$ is a trivial link with $k$ components, and
\[
	\{n\}_s=s^n-s^{-n}, \text{ and } [n]_s = \frac{\{n\}_s}{\{1\}_s}
\]
for an integer $n$.
We also define the quantum factorial $\{n\}_s !$ as usual.
Let $Q_i$ and $E_j$ be operators defined by
\begin{align*}
	(Q_i V_{\cL})(n_1,\ldots,n_k) &= s^{n_i}V_{\cL}(n_1,\ldots,n_k), \\
	(E_j V_{\cL})(n_1,\ldots,n_k) &= V_{\cL}(n_1,\ldots,n_j+1,\ldots,n_k).
\end{align*}
For each $j$, $Q_1,\ldots, Q_k$, and $E_j$ generate the algebra $\A^{j;k}$ consisting of $\mathbb{Z}[s,Q_1,\ldots,\allowbreak Q_k]$-coefficient polynomials
in $E_j$ with relations
\[
	E_j Q_j = sQ_j E_j  \text{ and } E_jQ_i = Q_i E_j, \text{ with } i \neq j.
\]
The $\hA$-polynomial $\hA_j(\cL)$ for $j$-th component of the link $\cL$ is defined as
a $\mathbb{Z}[s,Q_1,\ldots,\allowbreak Q_k]$-coefficient polynomial that generates the ideal of annihilating polynomials
of $V_{\cL}(n_1,\ldots,n_k)$ with respect to $n_j$
in a certain localization $\A^{j;k}_{\mathrm{loc}}$ of the algebra $\A^{j;k}$.
Then, there would be a similar relationship to the AJ conjecture
between the hyperbolic $A$-polynomial and the $\hA$-polynomial. 
More precisely, there would be a relationship, analogous to the AJ conjecture, between the ideal $I(\cL)$
and the ideal consisting of the annihilating polynomials of the colored Jones polynomial.
First, let us formulate the ideal consisting of the annihilating polynomials.
Let $\A_k$ be the algebra of noncommutative polynomials of the form
\[
	\sum _{j_1,\ldots,j_k}c_{j_1,\ldots,j_k}(s,Q_1,\ldots,Q_k)E_1^{j_1} \cdots E_k^{j_k},
\]
where $c_{j_1,\ldots,j_k}(s,Q_1,\ldots,Q_k) \in \mathbb{Q}(s,Q_1,\ldots,Q_k)$, with relations
\[
	E_i^l p(s,Q_1,\ldots,Q_k) = p(s,Q_1,\ldots,s^lQ_i,\ldots,Q_k) E_i^l \text{ and } E_iE_j = E_jE_i,
\]
where $p(s,Q_1,\ldots,Q_k) \in \mathbb{Q}(s,Q_1,\ldots,Q_k)$.
For the link $\cL$, the annihilating polynomials $\hA_j(\cL)$, with $j=1,\ldots,k$, are in the left ideal
\[
	\hI(\cL) = \{P\in \A_k \mid PV_{\cL}(n_1,\ldots,n_k) = 0\}
\]
of the algebra $\A_k$.
Let us call the ideal $\hI(\cL)$ the $\hA$-ideal of the link $\cL$.
Then, we conjecture the following relationship between the ideal $I(\cL)$ and the ideal $\hI(\cL)$.
\newtheorem*{uconj}{\textrm \textbf Conjecture~\ref{conj:linkAJ}}
\begin{uconj}
	For any link $\cL$, $ \varepsilon _s \hI(\cL) = I(\cL) $ holds.
\end{uconj}
Here, $ \varepsilon _s$ is the evaluation map at $s=1$.
In section \ref{sec:cal}, we find the $\hA$-ideal of the Hopf link $H$ and the Whitehead link $W$.
In fact, the $\hA$-polynomials are not enough as generators of the $\hA$-ideal in both cases of the Hopf link and the Whitehead link.
In both cases, there is a polynomial in $E_1$ and $E_2$ whose degree with respect to each variable is $1$.
We call this polynomial the $\hB$-polynomial.
We see that the $\hB$-polynomial and one of the $\hA$-polynomials generate the ideal $\hI(L)$, where $L=H$ or $W$.
See Theorem \ref{thm:genhIH} for the Hopf link, and Theorem \ref{thm:genhIW} for the Whitehead link.
In particular, for the Hopf link, we compute these polynomials directly.
Furthermore, we verify the following theorem:
\newtheorem*{uthm}{\textrm \textbf Theorem~\ref{thm:cjforH}}
\begin{uthm}
	$\varepsilon _s \hI(H) = I(H)$ holds.
\end{uthm}
For the Whitehead link, we compute annihilating polynomials by a creative telescoping method \cite{WZ}.
The colored Jones polynomial for the Whitehead link is
\begin{equation} \label{eq:VWFintro}
	V_W(m,n) = \sum _{i=0}^{\min(m,n)-1} F(m,n;i),
\end{equation}
where
\[
	F(m,n;i)=(-1)^{m+n}s^{-\frac{i^2+3i}{2}}\frac{\{m+i\}_s!\{n+i\}_s!\{i\}_s!}{\{1\}_s\{m-i-1\}_s!\{n-i-1\}_s!\{2i+1\}_s!}.
\]
Here, $\{n\}_s=s^n-s^{-n}$, and $\{n\}_s! = \{n\}_s \{n-1\}_s \cdots \{1\}_s$.
Let $E_1$, $Q_1$, $Q_2$, $\tE_1$ and $\tQ_1$ be operators defined by
\begin{align*}
	(E_1F)(m,n;i)&=F(m+1,n;i),\\
	(Q_1F)(m,n;i)&=s^mF(m,n;i),\\
	(E_2F)(m,n;i)&=F(m,n+1;i),\\
	(Q_2F)(m,n;i)&=s^nF(m,n;i),\\
	(\tE_1F)(m,n;i)&=F(m,n;i+1),\\
	(\tQ_1F)(m,n;i)&=s^iF(m,n;i).
\end{align*}
We can obtain an annihilating polynomial of the summand $F$ that does not have $\tQ_1$ as a variable from the equations
\[
	S_jE_j-R_j=0,\text{ with }j=1,2, \text{ and } \tilde{S}_1 \tE_1- \tilde{R}_1=0,
\]
where $R_j=R_j(s,Q_j,\tQ_1)$, $S_j=S_j(s,Q_j,\tQ_1)$, $\tilde{R}_1=\tilde{R}_1(s,Q_1,Q_2,\tQ_1)$ and $\tilde{S}_1=\tilde{S}_1(s,Q_1,Q_2,\tQ_1)$
are $\mathbb{Z}$-coefficient polynomials
\[
	\left.\frac{E_jF}{F}\right|_{\substack{(s^m,s^n,s^i) \\ =(Q_1,Q_2,\tQ_1)}}= \frac{R_j(s,Q_j,\tQ_1)}{S_j(s,Q_j,\tQ_1)}
\]
and
\[
	\left.\frac{\tE_1F}{F}\right|_{\substack{(s^m,s^n,s^i) \\ =(Q_1,Q_2,\tQ_1)}}= \frac{\tilde{R}_1(s,Q_1,Q_2,\tQ_1)}{\tilde{S}_1(s,Q_1,Q_2,\tQ_1)}.
\]
The creative telescoping is the method that gives an annihilating polynomial of the sum of the summand $F$
from the annihilating polynomial of the summand $F$.
Specifically, the equations
\[
	S_1E_1-R_1=0, \text{ and } \tilde{S}_1 \tE_1- \tilde{R}_1=0
\]
yield an annihilating polynomial $\Af_1(W)(s,E_1,Q_1,Q_2)$, and the equations
\[
	S_1E_1-R_1=0, \text{ and } S_2E_2-R_2=0
\]
yield the $\hB$-polynomial.
Here, the annihilating polynomial $\Af_1(W)(s,E_1,Q_1,Q_2)$ is $\mathbb{Q}(s,Q_1,Q_2)$-coefficient.
Generally, it is not certain whether any summand always yields an annihilating polynomial with the smallest degree.
In fact, we can factorize $\Af_1(W)$, but the right divisor in this factorization does not annihilate $V_W(m,n)$.
Thus, in this case, we obtain the smallest degree one.
Clearing the denominators of the coefficients of $\Af_1(W)$ yields the $\A$-polynomial $\hA_1(W)$.
Furthermore, we see that this polynomial $\Af_1(W)$ and the $\hB$-polynomial generate the $\hA$-ideal $\hI(W)$.
Additionally, we evaluate the $\hA$-ideal at $s=1$.
We see that the annihilating polynomial $\Af_1(W)(s,E_1,Q_1,Q_2)$ satisfies
\begin{align}
\begin{split} \label{eq:evsWintro}
	\varepsilon _s \Af_{1}&(W)(s,E_1,Q_1,Q_2)\\
	&= \frac{1-Q_1^2}{(1+Q_1^2)^2Q_1^2Q_2^2}(E_1+1)(E_1-Q_1^2)A_1(W)(E_1,Q_1^2,Q_2^2),
\end{split}
\end{align}
$A_1(W)(E_1,Q_1,Q_2)$ is the hyperbolic $A$-polynomial for the Whitehead link given in \cite{Mu} and Section \ref{sec:links}.
Similarly, evaluating the $\hB$-polynomial at $s=1$ yields the $B$-polynomial.
Therefore $ \varepsilon _s \hI(W) \subset I_W^\Phi $ holds,
where $\Phi$ is a potential function obtained from the summand $F$ in \eqref{eq:VWFintro}.
If the factors $E_1+1$ and $E_1-Q_1^2$ in \eqref{eq:evsWintro} correspond to some $SL(2,\mathbb{C})$-representations of the link group,
we can expect that $ \varepsilon _s \hI(W) = I(W) $ holds. \par
This paper is organized as follows:
In Section \ref{sec:prel}, we review the colored Jones polynomial and its potential function.
We also review the creative telescoping method, which is used to calculate the annihilating polynomial.
In Section \ref{sec:links}, we consider polynomials and ideals obtained from the colored Jones polynomial in both commutative and noncommutative cases.
We obtain the link version of the $A$-polynomial in the commutative case, and the link version of the $A_q$-polynomial in the noncommutative case.
The ideals $I(\cL)$ and $\hI(\cL)$ are also given in this section.
In Section \ref{sec:cal}, we calculate those annihilating polynomials for the Hopf link $H$ and the Whitehead link $W$.
Furthermore, by examining the relationship between these polynomials, we confirm that they generate the $\hA$-ideal $\hI(L)$,
where $L=H$ or $W$.
Additionally, we verify that $ \varepsilon _s \hI(W) \subset I_W^\Phi$ for the Whitehead link $W$
and formulate the conjecture that a similar relationship holds for general hyperbolic links.
In the Appendix, we find another relationship between $\Af_{1}(W)$, $\Af_{2}(W)$ and $\hB_{12}(W)$.
\par
\noindent \textit{Acknowledgments.} The author is grateful to Jun Murakami for his helpful comments.
\section{Preliminaries} \label{sec:prel}
\subsection{Colored Jones polynomial and the potential function} \label{sec:pf}
The colored Jones polynomial is derived from the following quantum group $\U_r,\ (r \in \mathbb{Z}_{>1})$ and its representation \cite{KM}:
Let $\U_r$ be an algebra generated by $X,\ Y,\ K,\ \overline{K}$ with the relations
\begin{align*}
	\overline{K}=K^{-1},\quad KX=sXK,\quad &KY=s^{-1}YK,\quad XY-YX=\frac{K^2-\overline{K}^2}{s-s^{-1}},\\
	&X^r=Y^r=0,\quad K^{4r}=1,
\end{align*}
where $s=e^{\frac{\pi \iu}{r}}$.
Let $V$ be an $n$-dimensional complex vector space, and let $\{e_{m},e_{m-1},\ldots,e_{-m}\}$ be a basis of $V$,
where $m$ is a half-integer satisfying $n=2m+1$.
We define the action of the algebra $\U_r$ on the vector space $V$ by
\begin{align*}
	Xe_i &= [m+i+1]_s e_{i+1}, \\
	Ye_i &= [m-i+1]_s e_{i-1}, \\
	Ke_i &= s^{i}e_i,
\end{align*}
where
\[
	[k]_s=\frac{s^k-s^{-k}}{s-s^{-1}}
\]
for an integer $k$.
The volume conjecture states that the colored Jones polynomial is related to the geometry of the knot complement.
\begin{conj}[Volume Conjecture \cite{MM}]
	For any knot $K$, the colored Jones polynomial $J_N(K;q)$ satisfies
	\[
		2 \pi \lim _{N \to \infty} \frac{\log |J_N(K;q=e^{\frac{2 \pi \iu}{N}})|}{N} = v_3 ||K||,
	\]
	where $v_3$ is the volume of the ideal regular tetrahedron in the three-dimensional hyperbolic space
	and $|| \cdot ||$ is the simplicial volume for the complement of $K$.
\end{conj}
Here, we normalize the colored Jones polynomial so that $J_N(\bigcirc;q)=1$ for the unknot $ \bigcirc $.
One idea to prove the conjecture is to use the saddle point method (see \cite{Yo}).
We assume that the colored Jones polynomial $J_N(K;q=e^{\frac{2 \pi \iu}{N}})$ is approximated as
\[
	J_N(K;q=e^{\frac{2 \pi \iu}{N}}) \sim \int _\Omega P_N e^{\frac{N}{2 \pi \iu}\Phi(w_1,\ldots,w_\nu)}dw_1 \cdots dw_\nu
\]
for a sufficiently large integer $N$.
Here, $ \Omega \subset \mathbb{C}^\nu $ is a region, and $P_N$ grows at most polynomially.
Then, the value at the saddle point of the function $\Phi(w_1,\ldots,w_\nu)$ contributes to the limit.
We call the function $\Phi$ a potential function.
In \cite{Sa}, we give the potential function $\Phi_D(\boldsymbol{a},w_1,\ldots,w_\nu)$ of $J_{\boldsymbol{n}}(\cL;q=e^{\frac{2 \pi \iu}{N}})$,
where $\cL$ is a $k$-component link with a diagram $D$ and $\boldsymbol{n}=(n_1,\ldots,n_k)$ is a tuple of colors. 
Here, $ \boldsymbol{a} = (a_1,\ldots,a_k)$ is a $k$-tuple of parameters, where
\[
	a_j = \lim _{N \to \infty} \frac{n_j}{N},\quad (j=1,\ldots, k).
\]
The colored Jones polynomials can be formulated using $R$-matrices,
and the potential function $\Phi_D$ is constructed by approximating the coefficients of the $R$-matrix
assigned to each crossing of the diagram $D$ with continuous functions.
We proved that the system of equations
\[
	\exp \left(w_{i}\frac{\partial \Phi _D}{\partial w_{i}}\right) = 1,\quad (i=1,\ldots,\nu)
\]
coincides with the gluing equation and that a saddle point
$(\sigma_1(\boldsymbol{a}),\ldots,\sigma_\nu(\boldsymbol{a}))$ gives the link complement a hyperbolic structure that is not necessarily complete.
We also proved that
\[
	\exp \left(\alpha_j \frac{\partial \Phi _D}{\partial \alpha_j}(\boldsymbol{\alpha},\sigma_1(\boldsymbol{\alpha}),\ldots,\sigma_\nu(\boldsymbol{\alpha}))\right)=l_j(\boldsymbol{\alpha})
\]
where $\alpha _j = e^{\pi \iu a_j}$, $\boldsymbol{\alpha} = (\alpha_1,\ldots,\alpha_k)$, and
$l_j(\boldsymbol{\alpha})$ is the dilation component of the $j$-th longitude.
In the knot case, the system of equations
\begin{equation} \label{eq:phihyp}
\begin{dcases}
	\exp \left(w_{i}\frac{\partial \Phi _D}{\partial w_{i}}\right) &= 1,\quad (i=1,\ldots,\nu) \\
	\exp \left(\alpha \frac{\partial \Phi _D}{\partial \alpha}\right) &= l^2,
\end{dcases}
\end{equation}
where $\Phi _D(\alpha,w_1,\ldots,w_\nu)$ is a potential function of the colored Jones polynomial $J_i(K;q=e^{\frac{2 \pi \iu}{N}})$ for a knot $K$,
and $l$ denotes the eigenvalue of the action of the longitude,
is a necessary condition where the knot complement admits a hyperbolic structure that is not necessarily complete.
Therefore, eliminating the variable $w_i$ from the above system of equations
would yield the factor of the $A$-polynomial $A_K(l,\alpha)$ for $K$ corresponding to nonabelian representations (see \cite{Gu,Hi,Y2}).

\subsection{Creative telescoping} \label{sec:cretele}
We can obtain an annihilating polynomial of $V_K(n)$ by creative telescoping \cite{WZ} (see also \cite{Ga,GL}). 
Let $K$ be a knot. We put
\[
	V_K(n)= \sum _{k_1,\ldots,k_\nu} F(n,k_1,\ldots,k_\nu).
\]
Let $E$, $Q$, $E_i$ and $Q_i$, with $i=1,\ldots,\nu $, be operators defined by
\begin{align*}
	(EF)(n,k_1,\ldots,k_\nu) &= F(n+1,k_1,\ldots,k_\nu),\\
	(QF)(n,k_1,\ldots,k_\nu) &= q^n F(n,k_1,\ldots,k_\nu),\\
	(\tE_i F)(n,k_1,\ldots,k_\nu) &= F(n,k_1,\ldots,k_i+1,\ldots,k_\nu),\\
	(\tQ_i F)(n,k_1,\ldots,k_\nu) &= q^{k_i}F(n,k_1,\ldots,k_\nu).
\end{align*}
Suppose that we find an annihilating polynomial $P(E,Q,\tE_1,\ldots,\tE_\nu)$ of $F$.
Expanding the polynomial $P$ around $\tE_i=1$, with $i=1,\ldots,\nu$, we have
\[
	P_0(E,Q)+\sum^\nu_{i=1}(E_i-1)R_i(E,Q,E_1,\ldots,E_\nu),
\]
where $P_0(E,Q)=P(E,Q,1,\ldots,1)$, and $R_i$, with $i=1,\ldots,\nu$, is a $\mathbb{Q}[q,Q]$-coefficient polynomial
in $E, E_1,\ldots,E_\nu$.
Putting $G_i = R_iF$, we have
\[
	P_0(E,Q)F(n,k_1,\ldots,k_\nu) + \sum^\nu_{i=1}(G_i(n,k_1,\ldots,k_i+1,\ldots,k_\nu)-G_i(n,k_1,\ldots,k_\nu))=0.
\]
Then, $P_0(E,Q)V_K(n)$ is a sum of some multisums.
Here, the number of variables of those multisums is one less than that of $V_K(n)$.
Repeating this process yields an annihilating polynomial of $V_K(n)$.
The remaining problem is to obtain the polynomial $P(E,Q,\tE_1,\ldots,\tE_\nu)$.
Since $F(n,k_1,\ldots,k_\nu)$ is a proper $q$-hypergeometric function (see \cite{WZ}), 
$\frac{EF}{F}$ and $\frac{\tE_iF}{F}$,with $i=1,\ldots,\nu $, are of the form
\begin{align*}
	\frac{EF}{F} &= \frac{R(q,q^n,q^{k_1},\ldots,q^{k_\nu})}{S(q,q^n,q^{k_1},\ldots,q^{k_\nu})},\\
	\frac{\tE_iF}{F} &= \frac{R_i(q,q^n,q^{k_1},\ldots,q^{k_\nu})}{S_i(q,q^n,q^{k_1},\ldots,q^{k_\nu})},
\end{align*}
where
\begin{align*}
	& R(q,q^n,q^{k_1},\ldots,q^{k_\nu}), S(q,q^n,q^{k_1},\ldots,q^{k_\nu}), R_i(q,q^n,q^{k_1},\ldots,q^{k_\nu}),\\
 	& \ \text{ and } S_i(q,q^n,q^{k_1},\ldots,q^{k_\nu}) \in \mathbb{Z}[q,q^n,q^{k_1},\ldots,q^{k_\nu}].
\end{align*}
This yields annihilating polynomials
\begin{align*}
	R(q,Q,\tQ_1,\ldots,\tQ_\nu)E &- S(q,Q,\tQ_1,\ldots,\tQ_\nu),\\
	R_i(q,Q,\tQ_1,\ldots,\tQ_\nu)\tE_i&- S_i(q,Q,\tQ_1,\ldots,\tQ_\nu),
\end{align*}
of $F(n,k_1,\ldots,k_\nu)$.
These polynomials generate the annihilating ideal $\mathrm{Ann}(F)$ of the summand $F$.
Therefore, if we can eliminate $\tQ_1,\ldots,\tQ_\nu$ from
\begin{align*}
	R(q,Q,\tQ_1,\ldots,\tQ_\nu)E - S(q,Q,\tQ_1,\ldots,\tQ_\nu) &= 0,\\
	R_i(q,Q,\tQ_1,\ldots,\tQ_\nu)\tE_i - S_i(q,Q,\tQ_1,\ldots,\tQ_\nu) &= 0,
\end{align*}
we obtain the annihilating polynomial of the summand $F$
with variables $E,Q,\tE_1,\ldots,\allowbreak \tE_\nu$.

\section{Polynomials and ideals derived from colored Jones polynomial} \label{sec:links}
In this section, we consider the link version of the $A$-polynomial and the $A_q$-polynomial.
As we mentioned in Introduction,
we use the colored Jones polynomial 
$V_{\cL}(n_1,\ldots,n_k) = V_{(n_1,\ldots,n_k)}(\cL;s)$ normalized as
\[
	V_{(n_1,\ldots,n_k)}(\bigcirc^k;s) = [n_1]_s \cdots[n_k]_s,
\]
for a $k$-component link $\cL$, where $s=q^{\frac{1}{2}}$.
\subsection{Commutative case} \label{ssec:linkscom}
Let $\cL=\cL_1 \cup \cdots \cup \cL_k$ be a $k$-component link,
and let $M$ be its exterior.
The boundary of $M$ is
\[
	\partial M = T_1 \cup \cdots \cup T_k,
\]
where $T_j$, with $j=1,\ldots,k$, is a torus that follows the $j$-th component $\cL_j$.
Let $\lambda _j$ and $\mu_j$, with $j=1,\ldots,k$, be the longitude and the meridian of the torus $T_j$.
We also let $\varrho: \pi_1(M) \to SL(2,\mathbb{C})$ be a representation of the link group.
Restricting the representation $\varrho$ to the torus $T_j$ yields an $SL(2,\mathbb{C})$-representation of $\pi_1(T_j)$.
Namely, the inclusion map $\iota_j:T_j \to M$ induces the map $(\iota_j)_*:\pi_1(T_j) \to \pi_1(M)$,
and we obtain a representation $\varrho_j:\pi_1(T_j) \to SL(2,\mathbb{C})$ by $\varrho_j = \varrho \circ (\iota_j)_*$.
As we already mentioned, the link version of the $A$-polynomial should express the conditions for the eigenvalues $l_j$'s and $ \alpha _j's$
of $\varrho_j(\lambda _j)$ and $\varrho_j(\mu _j)$.
For a $k$-component link, the $k$-tuple of $A$-polynomial was defined in \cite{Zh} and calculated for some links in \cite{Tr}.
Another kind of the $A$-polynomial associated with each component of a link is also investigated for the Whitehead link and the Borromean rings in \cite{Mu}.
The $A$-polynomial here is obtained from the colored Alexander polynomial by elimination.
To avoid confusion, we call the latter one the hyperbolic $A$-polynomial here. 
We can apply a similar argument to the colored Jones polynomial.
Suppose that $\cL=\cL_1 \cup \cdots \cup \cL_k$ is a hyperbolic link,
and let $\Phi(\alpha _1,\ldots, \alpha _k,w_1,\ldots,w_\nu)$ be the potential function of the colored Jones polynomial for $\cL$.
We can obtain a family of hyperbolic structures of $S^3 \setminus \cL$ with $\alpha_1,\ldots,\alpha_k$ as deformation parameters
by a solution of
\begin{equation} \label{eq:delPhilink}
	\begin{dcases}
		\exp \left(w_i \frac{\partial \Phi}{\partial w_i} \right)= 1, \quad (i = 1,\ldots,\nu)\\
		\exp \left(\alpha_j \frac{\partial \Phi}{\partial \alpha_j}\right) = l_j^2, \quad (j = 1,\ldots,k).
	\end{dcases}
\end{equation}
The left-hand sides of \eqref{eq:delPhilink} are rational functions of $w_i$'s and $\alpha_j$'s.
Therefore, clearing the denominators yields the system of algebraic equations
\[
	\begin{dcases}
		f_i(\alpha_1,\ldots,\alpha_k,w_1,\ldots,w_\nu)=0, \quad (i = 1,\ldots,\nu)\\
		g_j(l_j,\alpha_1,\ldots,\alpha_k,w_1,\ldots,w_\nu)=0, \quad (j = 1,\ldots,k),
	\end{dcases}
\]
where the left-hand sides are $\mathbb{Z}$-coefficient polynomials in $w_i$'s, $\alpha_j$'s and $l_j$'s.
Since the algebra $\mathbb{Q}(\alpha_1,\ldots, \allowbreak \alpha_k)[l_j]$ is a principal ideal domain as a ring for each $j$,
we can obtain the hyperbolic $A$-polynomial
\[
	A_j(\cL)(l_j,\alpha _1,\ldots,\alpha _k) \in \mathbb{Z}[\alpha _1,\ldots,\alpha _k,l_j]
\]
as a generator of the elimination ideal
\[
	\langle f_1,\ldots,f_\nu,g_j \rangle \cap \mathbb{Q}(\alpha_1,\ldots,\alpha_k)[l_j] \subset \mathbb{Q}(\alpha_1,\ldots,\alpha_k)[l_j]
\]
with coprime integer coefficients
up to multiplication by $\pm 1$. 
The hyperbolic $A$-polynomials belong to the ideal 
\[
	I_{\cL}^\Phi = \langle f_1,\ldots,f_\nu,g_1,\ldots,g_k \rangle \cap \mathbb{Q}(\alpha_1,\ldots,\alpha_k)[l_1,\ldots,l_k]
\]
of the algebra $\mathbb{Q}(\alpha_1,\ldots,\alpha_k)[l_1,\ldots,l_k]$.
A formula of the colored Jones polynomial for some knots and links,
such as the figure-eight knot and the Whitehead link, has only one summation index.
Moreover, we have a linear equation with respect to the variable
corresponding to the summation index as the saddle point equation.
In such a case, we can eliminate the variable by substitution.
\begin{ex}
The colored Jones polynomial $V_W(m,n)$ for the Whitehead link \cite{Ha} under the normalization of $V_W(m,n)=[m]_s[n]_s$ is
	\[
		V_W(m,n) = \sum _{i=0}^{\min(m,n)-1} F(m,n;i),
	\]
	where
	\[
		F(m,n;i)=(-1)^{m+n}s^{-\frac{i^2+3i}{2}}\frac{\{m+i\}_s!\{n+i\}_s!\{i\}_s!}{\{1\}_s\{m-i-1\}_s!\{n-i-1\}_s!\{2i+1\}_s!}.
	\]
The summand $F(m,n;i)|_{s=q^{\frac{1}{2}}}$ equals
\[
	(-1)^{m+n+i+1}q^{-mi-ni+\frac{i^2}{2}-\frac{m+n+i}{2}}\frac{(q)_{m+i}(q)_{n+i}(q)_{i}}{(1-q)(q)_{m-i-1}(q)_{n-i-1}(q)_{2i+1}},
\]
where $(q)_k=(1-q) \cdots (1-q^k)$ for an integer $k$ since
\[
	\{k\}! = (-1)^k q^{-\frac{1}{4}k(k+1)}(q)_k.
\]
For an integer $n_1,\ n_2$, and $N$ we have
\[
	e^{\frac{2 \pi \iu}{N}n_1n_2} = e^{\frac{N}{2 \pi \iu}\left(\frac{2 \pi \iu}{N}n_1\right)\left(\frac{2 \pi \iu}{N}n_2\right)}, \text{ and }
	(-1)^{n_1} = e^{\frac{N}{2 \pi \iu}\left(\pi \iu \frac{2 \pi \iu}{N}n_1\right)}.
\]
Moreover the direct calculation shows
\[
	\log (\xi_N)_{k} = \frac{N}{2 \pi \iu}\left( -\li(\xi^{k}_N)+\frac{\pi^2}{6}+o(1)\right),
\]
where $\xi_N = e^{\frac{2 \pi \iu}{N}}$, and
\[
	\li(z) = -\int^z_0 \frac{\log(1-x)}{x}dx.
\]
Therefore, we replace
\begin{itemize}
	\item $(-1)^{m+n+i}$ to $ \pi \iu (\log a + \log b + \log z)$,
	\item $-mi-ni+\frac{i^2}{2}$ to $ - \log a \log z - \log b \log z + \frac{1}{2}(\log z)^2$, and
	\item $(q)_{m+i}$ to $ - \li(a z) + \frac{\pi^2}{6}$.
\end{itemize}
We also replace other $q$-Pochhammer symbols in the same manner.
Then, we have a potential function
\begin{align*}
	\Phi(a,b,z) &= \pi \iu (\log a + \log b + \log z) - \log a \log z - \log b \log z + \frac{1}{2}(\log z)^2 \\
	 &- \li(a z)  - \li(b z) - \li(z) + \li (a z^{-1}) + \li(b z^{-1}) + \li(z^2).
\end{align*}
of $V_W(m,n)$. The equations
\[
	\exp \left(a \frac{\partial \Phi}{\partial a} \right) = l_1,\quad \exp \left(b \frac{\partial \Phi}{\partial b} \right) = l_2, \text{ and } \exp \left(z \frac{\partial \Phi}{\partial z} \right) = 1
\]
are respectively equivalent to
\begin{equation} \label{eq:delaW}
	a z -1 - l_1(z-a)=0,
\end{equation}
\begin{equation} \label{eq:delbW}
	b z -1 - l_2(z-b)=0,
\end{equation}
and
\begin{equation} \label{eq:delzW}
	(a z -1)(b z -1)(z -a)(z -b)-ab z(z^2-1)(z+1)=0. 
\end{equation}
From the equation \eqref{eq:delaW}, we have
\[
	z= \frac{l_1 a -1}{l_1-a}.
\]
Substituting it to the equation \eqref{eq:delzW}, we have
\begin{gather*}
	A_1(W)(l_1,a,b) = a^2 b l_1^3+(a^2 b^2-ab^2+a^2-2 ab-a+b)l_1^2\\
	\quad+(a^2 b-ab^2-2 ab+b^2-a+1)l_1+b = 0.
\end{gather*}
The left-hand side coincides with the polynomial obtained in \cite{Mu}.
Similarly, we can obtain the polynomial
\[
	A_2(W)(l_2,a,b) = A_1(W)(l_2,b,a)
\]
from the equations \eqref{eq:delbW} and \eqref{eq:delzW}.
Let $p_0(l_1,a,z)$ and $p_1(a,b,z)$ respectively be polynomials on the left-hand side of
the equation \eqref{eq:delaW} and \eqref{eq:delzW}.
Note that the polynomial on the left-hand side of the equation \eqref{eq:delbW} is $p_0(l_2,b,z)$.
We also let $\alpha$ and $ \beta $ be variables satisfing $a = \alpha^2$ and $b=\beta^2$.
Then the polynomials $A_1(W)$ and $A_2(W)$ belong to the ideal
\[
	I_W^\Phi = \langle p_0(l_1,\alpha^2,z),p_0(l_2,\beta^2,z),p_1(\alpha^2,\beta^2,z) \rangle \cap \mathbb{Q}(\alpha,\beta)[l_1,l_2]
\]
of the algebra $\mathbb{Q}(\alpha,\beta)[l_1,l_2]$.
\end{ex} 
\begin{rem}
	The variable $a$ here corresponds to the dilation component of the meridian.
	Thus, letting $\alpha$ be one of the eigenvalues of the action of the meridian, $a = \alpha^2$ holds.
\end{rem}
Recall that we only focus on a family of $SL(2,\mathbb{C})$-representations corresponding to hyperbolic structures.
In fact, for a hyperbolic knot, namely $k=1$, we cannot obtain the factor $l_1-1$ of the $A$-polynomial corresponding to abelian representations.
Therefore, we next formulate an ideal $I(\cL)$ that expresses
the conditions for defining an arbitrary family of $SL(2,\mathbb{C})$-representations parametrized by $\alpha_1,\ldots,\alpha_k$,
regardless of whether it corresponds to hyperbolic structures or not.
In the following formulation, the link $\cL$ need not be a hyperbolic link.
Let $O$ be the set of all possible orientations of the link $\cL$.
We fix an orientation $\sigma \in O$ of the link $\cL$.
The $j$-th longitude $\lambda_j$ is oriented in the same direction as the $j$-th component $\cL_j$,
and $j$-th meridian $\mu_j$ is chosen to have a right-handed orientation with respect to the longitude $\lambda_j$.
We start with a presentation of $\pi_1(M)$
\[
	\pi_1(M) = \langle x_1,\ldots,x_s \mid r_1,\ldots,r_t \rangle.
\]
We put
\[
	\varrho(x_a) = \left(\begin{matrix}
	x_{11}^{(a)} & x_{12}^{(a)} \\
	x_{21}^{(a)} & x_{22}^{(a)}
	\end{matrix}\right),\quad (a=1,\ldots,s).
\]
The entries satisfy
\begin{equation} \label{eq:detx2}
	x_{11}^{(a)}x_{22}^{(a)}-x_{12}^{(a)}x_{21}^{(a)}=1,\quad (a=1,\ldots,s).
\end{equation}
For $ \varrho $ to be an $SL(2,\mathbb{C})$-representation, the entries should also satisfy
\begin{equation} \label{eq:rhor2}
	\varrho(r_b) = I_2,\quad (b=1,\ldots,t).
\end{equation}
where $I_2$ is the $2 \times 2$ identity matrix.
\eqref{eq:rhor2} is a set of $4t$ algebraic equations.
For each $j$, since $\pi_1(T_j)$ is an abelian group,
$\varrho_j(\lambda_j)$ commutes with $\varrho_j(\mu_j)$.
Therefore, these can be simultaneously diagonalized, and furthermore,
we can choose elements in $SL(2,\mathbb{C})$ as the matrices for diagonalization.
Namely, there exists $g_j=(g^{(j)}_{\alpha \beta})$ such that
\begin{align*}
	\left(\begin{matrix}
	g_{22}^{(j)} & -g_{12}^{(j)} \\
	-g_{21}^{(j)} & g_{11}^{(j)}
	\end{matrix}\right) \varrho_j(\lambda_j) \left(\begin{matrix}
	g_{11}^{(j)} & g_{12}^{(j)} \\
	g_{21}^{(j)} & g_{22}^{(j)}
	\end{matrix}\right) = \left(\begin{matrix}
	l_j & 0 \\
	0 & \overline{l}_j
	\end{matrix}\right),\\
	\left(\begin{matrix}
	g_{22}^{(j)} & -g_{12}^{(j)} \\
	-g_{21}^{(j)} & g_{11}^{(j)}
	\end{matrix}\right) \varrho_j(\mu_j) \left(\begin{matrix}
	g_{11}^{(j)} & g_{12}^{(j)} \\
	g_{21}^{(j)} & g_{22}^{(j)}
	\end{matrix}\right) = \left(\begin{matrix}
	\alpha_j & 0 \\
	0 & \alpha_j^{-1}
	\end{matrix}\right),
\end{align*}
with
\begin{equation} \label{eq:detg2}
	g_{11}^{(j)}g_{22}^{(j)}-g_{12}^{(j)}g_{21}^{(j)}=1.
\end{equation}
Here,
\begin{equation} \label{eq:lj2}
	l_j \overline{l}_j =1
\end{equation}
holds.
Let $W_j$ and $W'_j$ respectively be words representing $(\iota_j)_*(\lambda_j)$ and $(\iota_j)_*(\mu_j)$.
Then, equalities
\begin{equation} \label{eq:diagl2}
	\left(\begin{matrix}
	g_{11}^{(j)} & g_{12}^{(j)} \\
	g_{21}^{(j)} & g_{22}^{(j)}
	\end{matrix}\right)\left(\begin{matrix}
	l_j & 0 \\
	0 & \overline{l}_j
	\end{matrix}\right)\left(\begin{matrix}
	g_{22}^{(j)} & -g_{12}^{(j)} \\
	-g_{21}^{(j)} & g_{11}^{(j)}
	\end{matrix}\right) = \varrho(W_j)
\end{equation}
and
\begin{equation} \label{eq:diaga2}
	\left(\begin{matrix}
	g_{11}^{(j)} & g_{12}^{(j)} \\
	g_{21}^{(j)} & g_{22}^{(j)}
	\end{matrix}\right)\left(\begin{matrix}
	\alpha_j & 0 \\
	0 & \alpha_j^{-1}
	\end{matrix}\right)\left(\begin{matrix}
	g_{22}^{(j)} & -g_{12}^{(j)} \\
	-g_{21}^{(j)} & g_{11}^{(j)}
	\end{matrix}\right) = \varrho(W'_j)
\end{equation}
hold for all $j=1,\ldots,k$.
Comparing entries of \eqref{eq:diagl2} and \eqref{eq:diaga2}, we have a total of $8k$
$\mathbb{Q}(\alpha_1,\ldots,\alpha_k)$-coefficient algebraic equations.
All equations \eqref{eq:detx2}, \eqref{eq:rhor2}, \eqref{eq:detg2}, \eqref{eq:lj2}, \eqref{eq:diagl2} and \eqref{eq:diaga2}
are equivalent to ones of the form (a $\mathbb{Q}(\alpha_1,\ldots,\alpha_k)$-coefficient polynomial)$=0$.
Let $S$ be the set of all those $\mathbb{Q}(\alpha_1,\ldots,\alpha_k)$-coefficient polynomials.
The set $S$ generates the ideal $\langle S \rangle$ in the algebra consisting of $\mathbb{Q}(\alpha_1,\ldots,\alpha_k)$-coefficient polynomials
in $x_{\alpha \beta}^{(a)}$'s, $g_{\alpha \beta}^{(j)}$'s, $l_j$'s and $\overline{l}_j$'s.
Then, we put
\[
	I^{\sigma}(\cL) = \langle S \rangle \cap \mathbb{Q}(\alpha_1,\ldots,\alpha_k)[l_1,\ldots,l_k].
\]
\begin{defi} \label{def:Aideal}
We define the ideal $I(\cL)$ as the intersection of the ideals $I^{\sigma}(\cL)$ over all $\sigma \in O$:
\[
	I(\cL) = \bigcap _{\sigma \in O} I^{\sigma}(\cL).
\]
\end{defi}
\begin{rem}
	Polynomials $g_j's$, where $g_j$ is a generator of the elimination ideal
	\[
		I(\cL) \cap \mathbb{Q}(\alpha_1,\ldots,\alpha_k)[l_j],\text{ with } j=1,\ldots,k,
	\]
	would not generate the ideal $I(\cL)$ in general.
	This is because the polynomials $g_j$'s represent the conditions that hold for each component
	and would not tell us whether a component has the same orientation as another component or not.
	Therefore, we introduce the $B$-polynomial for the Hopf link and the Whitehead link
	as a polynomial that would indicate the orientation of the components.
\end{rem}
\begin{ex}
For the Whitehead link, we can obtain
\[
	B_{12}(W)(l_1,l_2,a,b)=(l_1 a-1)(l_2-b)-(l_2 b-1)(l_1-a) = 0
\]
from \eqref{eq:delaW} and \eqref{eq:delbW}.
Let us call the polynomial $B_{12}(W)$ the $B$-polynomial for the Whitehead link $W$.
The polynomial $B_{12}(W)$ also belongs to the ideal $I_W^\Phi$.
The $B$-polynomial for the Hopf link will be given in Section \ref{sec:cal}.
\end{ex}

\subsection{Noncommutative case} \label{ssec:noncom}
On the other hand, we will define the $\hA$-polynomial for each component of a link $\cL=\cL_1 \cup \cdots \cup \cL_k$.
For $V_{\cL}(n_1,\ldots,n_k)$, we define operators $Q_i$ and $E_j$ by
\begin{align*}
	(Q_i V_{\cL})(n_1,\ldots,n_k) &= s^{n_i}V_{\cL}(n_1,\ldots,n_k), \\
	(E_j V_{\cL})(n_1,\ldots,n_k) &= V_{\cL}(n_1,\ldots,n_j+1,\ldots,n_k).
\end{align*}
For each $j$, $Q_1,\ldots, Q_k$, and $E_j$ generate the algebra $\A^{j;k}$ consisting of $\mathbb{Z}[s,Q_1,\ldots,\allowbreak Q_k]$-coefficient polynomials
in $E_j$ with relations
\[
	E_j Q_j = sQ_j E_j \text{ and } E_jQ_i = Q_i E_j, \text{ with } i \neq j.
\]
The argument in Section \ref{sec:cretele} implies the existence of the annihilating polynomial in the algebra $\A^{j;k}$:
\[
	\left(\sum _p c_p(s,Q_1,\ldots,Q_k)E_j^p \right)V_{\cL}(n_1,\ldots,n_k)=0,
\]
where $c_p(s,Q_1,\ldots,Q_k)$ is a polynomial with integer coefficients.
Then, we define the algebra $\A^{j;k}_{\mathrm{loc}}$, which is the localization of $\A^{j;k}$, by
\[
	\A^{j;k}_{\mathrm{loc}} = \left\{\sum^\infty_{p=0}a_p(s,Q_1,\ldots,Q_k)E_j^p \middle|
	\begin{array}{ll}
		a_p(s,Q_1,\ldots,Q_k) \in \mathbb{Q}(s,Q_1,\ldots,Q_k),\\
		a_p=0 \text{ for sufficiently large }p.
	\end{array} \right\},
\]
where the multiplication of monomials is given by
\begin{align*}
	&a(s,Q_1,\ldots,Q_k)E_j^p \cdot a'(s,Q_1,\ldots,Q_k)E_j^{p'}\\
	&\quad= a(s,Q_1,\ldots,Q_k)a'(s,Q_1,,\ldots,s^pQ_j \ldots,Q_k)E_j^{p+p'}.
\end{align*}
The $\hA$-polynomial $\hA_{j}(\cL)(s,E_j,Q_1,\ldots,Q_k)$ for the $j$-th component is a $\mathbb{Z}[s,Q_1,\ldots,\allowbreak Q_k]$-coefficient polynomial
that generate the annihilating ideal of $V_{\cL}(n_1,\ldots,n_k)$
\[
	I_j=\{P \in \A^{j;k}_{\mathrm{loc}} \mid PV_{\cL}(n_1,\ldots,n_k)=0\} \subset \A^{j;k}_{\mathrm{loc}}.
\]
We calculate annihilating polynomials of the colored Jones polynomial
for the Hopf link and the Whitehead link in Section \ref{sec:cal}.

\subsection{Annihilating ideal} \label{sec:AnotherAnn}
Let $\A_k$ be the algebra of noncommutative polynomials of the form
\[
	\sum _{j_1,\ldots,j_k}c_{j_1,\ldots,j_k}(s,Q_1,\ldots,Q_k)E_1^{j_1} \cdots E_k^{j_k},
\]
where $c_{j_1,\ldots,j_k}(s,Q_1,\ldots,Q_k) \in \mathbb{Q}(s,Q_1,\ldots,Q_k)$, with relations
\[
	E_i^l p(s,Q_1,\ldots,Q_k) = p(s,Q_1,\ldots,s^lQ_i,\ldots,Q_k) E_i^l \text{ and } E_iE_j = E_jE_i,
\]
where $p(s,Q_1,\ldots,Q_k) \in \mathbb{Q}(s,Q_1,\ldots,Q_k)$.
For the link $\cL$, the annihilating polynomials $\hA_j(\cL)$, with $j=1,\ldots,k$, are in the left ideal
\[
	\hI(\cL) = \{P\in \A_k \mid PV_{\cL}(n_1,\ldots,n_k) = 0\}
\]
of the algebra $\A_k$.
Let us call the ideal the $\hA$-ideal.
The fundamental question on the $\hA$-ideal is as follows:
\begin{question} \label{q:Asgen}
	Determine the generators of the $\hA$-ideal $\hI(\cL)$.
\end{question}
In fact, the $\hA$-polynomials are not enough as generators.
For the Hopf link and the Whitehead link, we need a $\mathbb{Q}(s,Q_1,Q_2)$-coefficient polynomial in $E_1$ and $E_2$
where the degree of each variable is $1$.
We call such a polynomial the $\hB$-polynomial.
In Section \ref{sec:cal}, we give the $\hB$-polynomial for the Hopf link and the Whitehead link,
and find the relationship between the $\hA$-polynomials and the $\hB$-polynomial.

\section{Calculations of the annihilating polynomials} \label{sec:cal}
\subsection{Hopf link}
Let us calculate the annihilating polynomials of 
the colored Jones polynomial $V_H(m,n)=V_{(m,n)}(H;s)$ of the Hopf link $H$,
though this is not a hyperbolic link.
The colored Jones polynomial $V_H(m,n)$ for the Hopf link $H$ is
	\[
		V_H(m,n) = [mn]_s.
	\]
\subsubsection{Annihilating polynomial with respect to $m$}
	
	By using the equality
	\[
		\{(m+2)n\}_s - (s^n+s^{-n})\{(m+1)n\}_s + \{mn\}_s=0,
	\]
	we obtain
	\[
		\hA_{1}(H)(s,E_1,Q_1,Q_2)=E_1^2-(Q_2+Q_2^{-1})E_1 + 1
	\]
	as an annihilating polynomial of $V_H(m,n)$,
	where $E_1$, $Q_1$ and $Q_2$ are operators defined as follows:
	\begin{align*}
		(E_1 V_H)(m,n) &= V_H(m+1,n),\\
		(Q_1 V_H)(m,n) &= s^m V_H(m,n),\\
		(E_2 V_H)(m,n) &= V_H(m,n+1),\\
		(Q_2 V_H)(m,n) &= s^n V_H(m,n).
	\end{align*}
	Since $E_1$ and $Q_2$ commute, we can factorize the polynomial $\hA_{1}(H)$ as
	\[
		(E_1-Q_2)(E_1-Q_2^{-1}) = (E_1-Q_2^{-1})(E_1-Q_2).
	\]
	The factors $E_1-Q_2$ and $E_1-Q_2^{-1}$, however, do not annihilate $V_H(m,n)$.
	Thus, the polynomial $\hA_{1}(H)$ is the annihilating polynomial of $V_H(m,n)$ with the least $E_1$-degree.
	Similarly, we have
	\[
		\hA_{2}(H)(s,E_2,Q_1,Q_2)=E_2^2-(Q_1+Q_1^{-1})E_2 + 1
	\]
\subsubsection{$\hB$-polynomial of $V_H(m,n)$}
Since $\{1\}_s$ is a constant,
it suffices to find the recurrence relation of $\{mn\}_s$.
Acting $E_1$ and $E_2$, we have
\begin{equation} \label{eq:Emmn}
	E_1 \{mn\}_s = \{(m+1)n\}_s = s^{(m+1)n}-s^{-(m+1)n},
\end{equation}
and
\[
	E_2 \{mn\}_s = \{m(n+1)\}_s = s^{m(n+1)}-s^{-m(n+1)}.
\]
Multiplying $\{m\}_s=s^m-s^{-m}$ to \eqref{eq:Emmn} from the left, we have
\begin{equation} \label{eq:EmJHa}
	(s^m-s^{-m})E_1\{mn\}_s = s^{mn+m+n}-s^{-mn+m-n}-s^{mn-m+n}+s^{-mn-m-n}.
\end{equation}
Similarly,
\begin{equation} \label{eq:EnJHa}
	(s^n-s^{-n})E_2\{mn\}_s = s^{mn+m+n}-s^{-mn-m+n}-s^{mn+m-n}+s^{-mn-m-n}.
\end{equation}
Subtracting \eqref{eq:EnJHa} from \eqref{eq:EmJHa}, we have
\begin{align*}
	&\{(s^m-s^{-m})E_1-(s^n-s^{-n})E_2\}\{mn\}_s\\
	&\quad= s^{-mn-m+n}+s^{mn+m-n}-s^{-mn+m-n}-s^{mn-m+n}\\
	&\quad=(s^{m-n}-s^{n-m})(s^{mn}-s^{-mn}) = (s^{m-n}-s^{n-m})\{mn\}_s.
\end{align*}
Therefore,
\begin{equation}\label{eq:hlann}
	\hB_{12}(H)(s,E_1,E_2,Q_1,Q_2)=a(Q_1)E_1 - a(Q_2)E_2 + a(Q_2Q_1^{-1})
\end{equation}
is an annihilating polynomial of $V_H(m,n)$, where $a(x)=x-x^{-1}$.
This has less degree with respect to $E_1$ than $\hA_{1}(H)$ and is alternating under the transposition of $1$ and $2$.

\subsubsection{Generators of the $\hA$-ideal $\hI(H)$}
When we `divide' $\hA_{1}(H)(s,E_1,Q_1,Q_2)$ by \eqref{eq:hlann}, we have
\begin{align}
\begin{split} \label{eq:hlanndiv}
	&\hA_{1}(H)(s,E_1,Q_1,Q_2)\\
	&\quad=\left(\frac{1}{a(sQ_1)}E_1+p(s,E_2,Q_1,Q_2)\cdot \frac{1}{a(Q_1)}\right)\hB_{12}(H)(s,E_1,E_2,Q_1,Q_2)\\
	&\quad \quad+\frac{a(Q_2)a(sQ_2)}{a(Q_1)a(sQ_1)}\hA_{2}(H)(s,E_2,Q_1,Q_2),
\end{split}
\end{align}
where
\[
	p(s,E_2,Q_1,Q_2)=\frac{a(Q_2)}{a(sQ_1)}E_2-Q_2-Q_2^{-1}-\frac{1}{a(sQ_1)}a \left( \frac{Q_2}{sQ_1}\right).
\]
Multiplying $a(Q_1)a(sQ_1)$ to \eqref{eq:hlanndiv} from the left, we have
\begin{align}
\begin{split} \label{eq:hlanndiv2}
	&a(Q_1)a(sQ_1)\hA_{1}(H)(s,E_1,Q_1,Q_2)\\
	&\quad -a(Q_2)a(sQ_2)\hA_{2}(H)(s,E_2,Q_1,Q_2)\\
	&\quad \quad=\left\{a(Q_1)E_1+a(Q_2)E_2-a(sQ_1Q_2)\right\}\hB_{12}(H)(s,E_1,E_2,Q_1,Q_2).
\end{split}
\end{align}
More generally, when dividing any polynomial in the $\hA$-ideal $\hI(H)$ by $\hB_{12}(H)$ as polynomials in $E_1$,
the remainder is the polynomial in $s,\ E_2,\ Q_1$ and $Q_2$ since $\deg(\hB_{12}(H);E_1)=1$.
This implies the following theorem: 
\begin{thm} \label{thm:genhIH}
	$\hI(H) = \langle \hA_{1}(H),\ \hB_{12}(H)\rangle $ holds.
\end{thm}
\subsubsection{Evaluation at $s=1$}
Evaluating $\hA_1(H)$ at $s=1$, we have
\[
	\varepsilon _s \hA_1(H) = E_1^2-(Q_2+Q_2^{-1})E_1 + 1.
\]
Evaluating the $\hB$-polynomial at $s=1$, we have
\[
	B_{12}(H)(E_1,E_2,Q_1,Q_2)=a(Q_1)E_1 - a(Q_2)E_2 + a(Q_2Q_1^{-1}).
\]
Let us call this polynomial the $B$-polynomial of the Hopf link.
Therefore,
\[
	\varepsilon _s \hI(H)=\langle \varepsilon _s \hA_1(H), B_{12}(H) \rangle.
\]
\begin{thm} \label{thm:cjforH}
	$\varepsilon _s \hI(H) = I(H)$ holds.
\end{thm}
\begin{proof}
The Hopf link $H$ has two possible orientations: one in which the second component is right-handed relative to the first,
and the other in which it is left-handed.
We denote the former by $+$ and the latter by $-$.
Let $ \lambda _j$ be the longitude, and let $ \mu _j$ be the meridian of the $j$-th component, with $j=1,2$.
Since the link group of the Hopf link is an abelian group generated by two meridians $\mu_1$ and $\mu_2$,
the link group has only abelian representations.
Commutativity implies that the representation of the link group can be assumed to have the following form:
\[
	\lambda _j \mapsto \left(
	\begin{array}{cc}
		E_j&	0 \\
		0&	E_j^{-1}
	\end{array}
	\right),\quad \mu _j \mapsto \left(
	\begin{array}{cc}
		Q_j&	0 \\
		0&	Q_j^{-1}
	\end{array}
	\right),
\]
with $j=1,2$.
Under the orientation $+$, $\lambda_1 = \mu_2$ and $\lambda_2 = \mu_1$ hold.
These imply $E_1=Q_2$ and $E_2=Q_1$.
Thus
\[
	I^{+}(H) = \langle E_1-Q_2,\ E_2-Q_1 \rangle.
\]
On the other hand, $\lambda_1 = \mu_2^{-1}$ and $\lambda_2 = \mu_1^{-1}$ hold under the orientation $-$.
These imply $E_1=Q_2^{-1}$ and $E_2=Q_1^{-1}$.
Thus
\[
	I^{-}(H) = \langle E_1-Q_2^{-1},\ E_2-Q_1^{-1} \rangle.
\]
Therefore the ideal $I(H)$ is
\begin{gather*}
	I(H) = I^{+}(H) \cap I^{-}(H) \\
	=\langle \varepsilon _s \hA_1(H),\ \varepsilon _s \hA_2(H),\ (E_1-Q_2)(E_2-Q_1^{-1}),\ (E_1-Q_2^{-1})(E_2-Q_1)\rangle.
\end{gather*}
Subtracting the fourth polynomial from the third yields the $B$-polynomial $B_{12}(H)$.
Thus,
\[
	I(H)=\langle \varepsilon _s \hA_1(H),\ \varepsilon _s \hA_2(H),\ B_{12}(H),\ (E_1-Q_2^{-1})(E_2-Q_1)\rangle.
\]
The fourth polynomial satisfies
\[
	(E_1-Q_2^{-1})(E_2-Q_1) = \frac{a(Q_1)}{a(Q_2)}\varepsilon _s \hA_1(H) -\frac{1}{a(Q_2)}(E_1-Q_2^{-1})B_{12}(H),
\]
and is therefore redundant for the set of generators.
Therefore,
\[
	I(H)=\langle \varepsilon _s \hA_1(H),\ \varepsilon _s \hA_2(H),\ B_{12}(H)\rangle.
\]
Similar to the non-commutative case, $\hA_2(H)$ is redundant as a generator, and therefore $\varepsilon _s \hI(H) = I(H)$ holds.
\end{proof}
\begin{rem}
The roots of $\varepsilon _s \hA_1(H)$ are $E_1=Q_2^{\pm 1}$.
Similarly, we have $E_2 = Q_1^{\pm 1}$ from $\varepsilon _s \hA_2(H) = 0 $.
Therefore, $\varepsilon _s \hA_j(H)$ yields the possible values of the eigenvalue $E_j$ associated with the $j$-th component, with $j=1,2$.
Among the possible combinations of roots, $(E_1,E_2)=(Q_2,Q_1)$, and $(Q_2^{-1},Q_1^{-1})$ satisfy$B_{12}(H)=0$.
Therefore, $B$-polynomial represents the condition for the orientations of the components to be consistent.
\end{rem}

\subsection{Whitehead link}
Let us calculate the annihilating polynomials of 
the colored Jones polynomial $V_W(m,n)=V_{(m,n)}(W;s)$ of the Whitehead link $W$.
Recall that the colored Jones polynomial $V_W(m,n)$ for the Whitehead link is
	\[
		V_W(m,n) = \sum _{i=0}^{\min(m,n)-1} F(m,n;i),
	\]
	where
	\[
		F(m,n;i)=(-1)^{m+n}s^{-\frac{i^2+3i}{2}}\frac{\{m+i\}_s!\{n+i\}_s!\{i\}_s!}{\{1\}_s\{m-i-1\}_s!\{n-i-1\}_s!\{2i+1\}_s!}.
	\]
\subsubsection{Annihilating polynomial with respect to $m$} \label{sssec:A1WL}
	We define operators $E_1$, $Q_1$, $Q_2$, $\tE_1$ and $\tQ_1$ as follows:
	\begin{align*}
		(E_1F)(m,n;i)&=F(m+1,n;i),\\
		(Q_1F)(m,n;i)&=s^mF(m,n;i),\\
		(Q_2F)(m,n;i)&=s^nF(m,n;i),\\
		(\tE_1F)(m,n;i)&=F(m,n;i+1),\\
		(\tQ_1F)(m,n;i)&=s^iF(m,n;i).
	\end{align*}
	Then, $\frac{E_1F}{F}$ and $\frac{\tE_1F}{F}$ are
	\begin{align*}%
		\frac{E_1F}{F} &= -\frac{(s^{2(m+i+1)}-1)}{s(s^{2m}-s^{2i})}=\frac{(1-s^{2(m+i+1)})s^m}{(s^{2m}-s^{2i})s^{m+1}},\\
		\frac{\tE_1F}{F} &= s^{2i-2m-2n+2}\frac{(s^{2(m+i+1)}-1)(s^{2(m-i-1)}-1)(s^{2(n+i+1)}-1)(s^{2(n-i-1)}-1)}{(s^{2(2i+3)}-1)(s^{2(i+1)}+1)}.
	\end{align*}
Therefore, we can obtain an annihilating polynomial of $F(m,n;i)$ by eliminating $\tQ_1$ from
	\begin{align}
		&(E_1-s^2Q_1^2)\tQ_1^2Q_1 = (Q_1^2E_1-1)Q_1, \label{eq:whlQ2}\\
		\begin{split} \label{eq:whlE1}
		&s^2Q_1^2Q_2^2(s^6 \tQ_1^4-1)(s^2 \tQ_1^2+1)\tQ_1^2 \tE_1\\
		&\ - (s^2Q_1^2 \tQ_1^2-1)(Q_1^2-s^2 \tQ_1^2)(s^2Q_2^2 \tQ_1^2-1)(Q_2^2-s^2 \tQ_1^2) = 0.
		\end{split}
	\end{align}
Since the left-hand side of \eqref{eq:whlE1} has a degree of $8$ with respect to $\tQ_1$,
eliminating $\tQ_1$ using \eqref{eq:whlQ2} requires a polynomial with a degree of at least $4$ with respect to $E_1$.
Let us find this polynomial.
By multiplying $(s^4Q_1^2Q_2^2 \tQ_1^4)^{-1}Q_1$ to \eqref{eq:whlE1} from the left, we have
\begin{align} 
\begin{split}\label{eq:whlE1a}
	&\left\{ \tE_1(s^2 \tQ_1^4+s^2 \tQ_1^2-1- \tQ_1^{-2}) \right.\\
	&\quad -s^4 \tQ_1^4+s^2(Q_1^2+Q_2^2+Q_1^{-2}+Q_2^{-2})\tQ_1^2 \\
	&\quad \quad -(Q_1^{2}Q_2^{2}+Q_1^{2}Q_2^{-2}+2+Q_1^{-2}Q_2^{2}+Q_1^{-2}Q_2^{-2})\\
	&\quad \quad \quad \left.+s^{-2}(Q_1^2+Q_2^2+Q_1^{-2}+Q_2^{-2}) \tQ_1^{-2} -s^{-4} \tQ_1^{-4}\right\}Q_1=0.
\end{split}
\end{align}
By multiplying $\tQ_1^{-2}$ to \eqref{eq:whlQ2} from the left, we have
\begin{equation} \label{eq:whlQ2inv}
	(Q_1^2E_1-1)\tQ_1^{-2}Q_1 = (E_1-s^2Q_1^2)Q_1.
\end{equation}
We put $\Phi_k(Q_1)=1-s^kQ_1^4$ for an integer $k$.
Concerning $E_1-s^2Q_1^2$ and $Q_1^2E_1-1$,
\begin{equation} \label{eq:sfact}
	u_1(s,E_1,Q_1)(Q_1^2E_1-1) = v_1(s,E_1,Q_1)(E_1-s^2Q_1^2)
\end{equation}
holds, where
\[
	u_1(s,E_1,Q_1) = \frac{1}{\Phi_6(Q_1)}E_1 - \frac{s^2Q_1^2}{\Phi_2(Q_1)},
\]
and
\[
	v_1(s,E_1,Q_1) = \frac{s^2Q_1^2}{\Phi_6(Q_1)}E_1 - \frac{1}{\Phi_2(Q_1)}.
\]
By multiplying $u_1(s,E_1,Q_1)\tQ_1^2$
to \eqref{eq:whlQ2} from the left and using \eqref{eq:sfact},
we have
\begin{align}
\begin{split} \label{eq:whlQ4}
	&u_1(s,E_1,Q_1)(E_1-s^2Q_1^2)\tQ_1^4Q_1\\
	&\quad =v_1(s,E_1,Q_1)(Q_1^2E_1-1)Q_1.
\end{split}
\end{align}
We also have
\begin{align}
\begin{split} \label{eq:whlQ4inv}
	&v_1(s,E_1,Q_1)(Q_1^2E_1-1)\tQ_1^{-4}Q_1\\
	&\quad =u_1(s,E_1,Q_1)(E_1-s^2Q_1^2)Q_1
\end{split}
\end{align}
by multiplying $Q_1^{-4}$ to \eqref{eq:whlQ4} from the left.
Noting that the polynomial
\[
	Y(s,E_1,Q_1) = u_2(s,E_1,Q_1)u_1(s,E_1,Q_1)(E_1-s^2Q_1^2),
\]
where
\[
	u_2(s,E_1,Q_1) = \frac{s^{12}Q_1^4}{\Phi_{10}(Q_1)\Phi_{12}(Q_1)}E_1^2 - \frac{s^4(1+s^2)Q_1^2}{\Phi_6(Q_1)\Phi_{10}(Q_1)}E_1 + \frac{1}{\Phi_4(Q_1) \Phi_6(Q_1)}
\]
is also factorized as
\[
	Y(s,E_1,Q_1) = v_2(s,E_1,Q_1)v_1(s,E_1,Q_1)(Q_1^2E_1-1),
\]
where
\[
	v_2(s,E_1,Q_1) = \frac{1}{\Phi_{10}(Q_1)\Phi_{12}(Q_1)}E_1^2 - \frac{s^2(1+s^2)Q_1^2}{\Phi_6(Q_1)\Phi_{10}(Q_1)}E_1 + \frac{s^4Q_1^4}{\Phi_4(Q_1) \Phi_6(Q_1)},
\]
we can immediately eliminate $\tQ_1^4$, $\tQ_1^{-4}$, $(Q_2^2+Q_2^{-2})\tQ_1^2$ and $(Q_2^2+Q_2^{-2})\tQ_1^{-2}$ from \eqref{eq:whlE1a}
by multiplying $Y(s,E_1,Q_1)$ from the left and using \eqref{eq:whlQ2}, \eqref{eq:whlQ2inv}, \eqref{eq:whlQ4} and \eqref{eq:whlQ4inv}.
Regarding $(Q_1^2+Q_1^{-2})\tQ_1^2$, 
\begin{align*}
	(E_1-s^2Q_1^2)&Q_1^2 \tQ_1^2 = s^2Q_1^2(E_1-Q_1^2)\tQ_1^2 \\
	&=s^2Q_1^2(E_1-s^2Q_1^2)\tQ_1^2 + s^2(s^2-1)Q_1^4 \tQ_1^2,
\end{align*}
and
\begin{align*}
	(E_1-s^2Q_1^2)&Q_1^{-2} \tQ_1^2 = s^{-2}Q_1^{-2}(E_1-s^4Q_1^2) \tQ_1^2 \\
	&=s^{-2}Q_1^{-2}(E_1-s^2Q_1^2)\tQ_1^2 + (1-s^2)\tQ_1^2
\end{align*}
hold. Summing up these equalities, we have
\begin{align*}
	(E_1-s^2Q_1^2)&(Q_1^2+Q_1^{-2})\tQ_1^2 \\
	&=(s^2Q_1^2+s^{-2}Q_1^{-2})(E_1-s^2Q_1^2)\tQ_1^2 + (1-s^2)(1-s^2Q_1^4)\tQ_1^2.
\end{align*}
Noting that
\[
	u_1(s,E_1,Q_1)=\frac{1}{\Phi_6(Q_1)}E_1 - \frac{s^2Q_1^2}{\Phi_2(Q_1)} = (E_1-s^2Q_1^2) \frac{1}{\Phi_2(Q_1)},
\]
we obtain
\begin{align*}
	&u_1(s,E_1,Q_1)(E_1-s^2Q_1^2)(Q_1^2+Q_1^{-2})\tQ_1^2\\
	& \quad =u_1(s,E_1,Q_1)(s^2Q_1^2+s^{-2}Q_1^{-2})(E_1-s^2Q_1^2)\tQ_1^2 \\
	&\quad \quad+ (1-s^2)(E_1-s^2Q_1^2)\tQ_1^2.
\end{align*}
Similarly, we obtain
\begin{align*}
	&v_1(s,E_1,Q_1)(Q_1^2E_1-1)(Q_1^2+Q_1^{-2})\tQ_1^{-2}\\
	& \quad =v_1(s,E_1,Q_1)(s^2Q_1^2+s^{-2}Q_1^{-2})(Q_1^2E_1-1)\tQ_1^{-2} \\
	&\quad \quad+ s^{-2}(1-s^2)Q_1^{-2}(Q_1^2E_1-1)\tQ_1^{-2}.
\end{align*}
Therefore, we obtain an annihilating polynomial $P_W(s,E_1,Q_1,Q_2,\tE_1)$ of $F(m,n;i)$
\begin{align*}
	&P_W(s,E_1,Q_1,Q_2,\tE_1)\\
	&\quad=\left\{(s^2 \tE_1-s^4)u_2(s,E_1,Q_1)v_1(s,E_1,Q_1)(Q_1^2E_1-1)\right. \\
	&\quad + s^2(Q_2^2+ \tE_1+Q_2^{-2})u_2(s,E_1,Q_1)u_1(s,E_1,Q_1)(Q_1^2E_1-1)\\
	&\quad +s^2u_2(s,E_1,Q_1)u_1(s,E_1,Q_1)(s^2Q_1^2+s^{-2}Q_1^{-2})(Q_1^2E_1-1)\\
	&\quad +(s^2-s^4)u_2(s,E_1,Q_1)(Q_1^2E_1-1)\\
	&\quad -Y(s,E_1,Q_1)(Q_1^{2}Q_2^{2}+Q_1^{2}Q_2^{-2}+2+ \tE_1+Q_1^{-2}Q_2^{2}+Q_1^{-2}Q_2^{-2})\\
	&\quad +s^{-2}(Q_2^2-s^2 \tE_1+Q_2^{-2})v_2(s,E_1,Q_1)v_1(s,E_1,Q_1)(E_1-s^2Q_1^2)\\
	&\quad +s^{-2}v_2(s,E_1,Q_1)v_1(s,E_1,Q_1)(s^2Q_1^2+s^{-2}Q_1^{-2})(E_1-s^2Q_1^2)\\
	&\quad +(s^{-4}-s^{-2})v_2(s,E_1,Q_1)Q_1^{-2}(E_1-s^2Q_1^2)\\
	&\quad -s^{-4}\left. v_2(s,E_1,Q_1)u_1(s,E_1,Q_1)(E_1-s^2Q_1^2)\right\}Q_1.
\end{align*}
We put
\[
	\tilde{P}_W^0(s,E_1,Q_1,Q_2) = P_W(s,E_1,Q_1,Q_2,1).
\]
We let the expansion of $P_W(s,E_1,Q_1,Q_2,\tE_1)$ around $\tE_1=1$ be
\[
	P_W(s,E_1,Q_1,Q_2,\tE_1) = \tilde{P}_W^0(s,E_1,Q_1,Q_2)+(\tE_1-1)R_W(s,E_1,Q_1),
\]
and put $G(m,n;i)=R_W(s,E_1,Q_1)F(m,n;i)$.
The polynomial $R_1(s,E_1,Q_1)$ is equal to
\[
	R_W(s,E_1,Q_1) = Y(s,E_1,Q_1)(s^2 \tQ_1^4+s^2 \tQ_1^2-1-\tQ_1^{-2})Q_1.
\]
Putting $F'(m,n;i)=(s^2 \tQ_1^4+s^2 \tQ_1^2-1- \tQ_1^{-2})Q_1 F(m,n;i)$, we have
\begin{align*}
	F'(m,n;i)&=(-1)^{m+n}s^{-\frac{i^2+3i}{2}+m}(s^{2+4i}+s^{2+2i}-1-s^{-2i})\\
	&\times \frac{\{m+i\}_s!\{n+i\}_s!\{i\}_s!}{\{1\}_s\{m-i-1\}_s!\{n-i-1\}_s!\{2i+1\}_s!}.
\end{align*}
Especially, we have%
\[
	F'(m,n,0)=(-1)^{m+n}\times 2(s^2-1)s^{m}[m]_s[n]_s.
\]
When we put
\[
	Y(s,E_1,Q_1) = \sum^{4}_{\lambda = 0}c_ \lambda(s,Q_1)E_1^\lambda,
\]
$G(m,n;i)$ is
\[
	G(m,n;i) = \sum^{4}_{\lambda = 0}c_ \lambda(s,s^m)F'(m+\lambda,n;i).
\]
Since the degree of $\tilde{P}_W^0(s,E_1,Q_1,Q_2)$ with respect to $E_1$ is $4$, summing up
\[
	\tilde{P}_W^0(s,E_1,Q_1,Q_2)F(m,n;i)=-(\tE_1-1)G(m,n;i)
\]
with $i=0,\ldots,\min(m+4,n)-1$, we have
\[
	\tilde{P}_W^0(s,E_1,Q_1,Q_2)V_W(m,n) = -\left( G(m,n,\min(m+4,n))-G(m,n,0)\right).
\]
Here, $G(m,n,\min(m+4,n)) =0$ and
\begin{align*}
	G(m,n,0)&=\sum^{4}_{\lambda = 0}c_ \lambda(s,s^m)F'(m+\lambda,n,0)\\
	&=(-1)^{m+n} \times 2(1-s^2)\frac{(1+s^{2m+4})[n]_s}{(1-s^{2+2m})(1-s^{6+2m})}.
\end{align*}
Since $G(m,n,0)$ is annihilated by
\[
	P_W^1(s,E_1,Q_1)=(E_1+1)\times \frac{(1-s^2Q_1^2)(1-s^6Q_1^2)}{1+s^4Q_1^2},
\]
we obtain an annihilating polynomial
\[
	P_W^1(s,E_1,Q_1)\tilde{P}_W^0(s,E_1,Q_1,Q_2)
\]
of $V_W(m,n)$.
Putting $\tilde{P}_W^0(s,E_1,Q_1,Q_2)=P_W^0(s,E_1,Q_1,Q_2)Q_1$, we have
\[
	P_W^1(s,E_1,Q_1)\tilde{P}_W^0(s,E_1,Q_1,Q_2)=Q_1 P_W^1(s,sE_1,Q_1)P_W^0(s,sE_1,Q_1,Q_2).
\]
$P_W^1(s,sE_1,Q_1)P_W^0(s,sE_1,Q_1,Q_2)$ is also an annihilating polynomial of $V_W(m,n)$.
We put
\[
	\Af_{1}(W)(s,E_1,Q_1,Q_2) = P_W^1(s,sE_1,Q_1)P_W^0(s,sE_1,Q_1,Q_2).
\]
Since the factor $P_W^0(s,sE_1,Q_1,Q_2)$ does not annihilate $V_W(m,n)$,
the polynomial $\Af_{1}(W)(s,E_1,Q_1,Q_2)$ is the annihilating polynomial with the smallest $E_1$-degree.
The coefficients of the polynomial $\Af_{1}(W)$ are in $\mathbb{Q}(s,Q_1,Q_2)$.
Clearing the denominators of the coefficients in the polynomial $\Af_1(W)$ yields the $\hA$-polynomial $\hA_1(W)$.
\begin{rem} \label{rem:VvsJ}
Let us observe the annihilating polynomial of the colored Jones polynomial $J_{(n_1,\ldots,n_k)}(\cL;s)$ with the normalization
$J_{(n_1,\ldots,n_k)}(\bigcirc^k;s) = 1$.
Since
\[
	J_W(m,n)= J_{(m,n)}(W;s)= \frac{V_W(m,n)}{[m]_s[n]_s},
\]
we have an annihilating polynomial of $J_W(m,n)$
\[
	P_W^1(s,E_1,Q_1)P_W^0(s,E_1,Q_1,Q_2)(Q_1^2-1).
\]
This is a polynomial in $s^2$, $E_1$, $Q_1^2$, and $Q_2^2$.
\end{rem}

\subsubsection{$\hB$-polynomial of $V_W(m,n)$} 
We eliminate $\tQ_1^2$ from \eqref{eq:whlQ2} and
\begin{equation} \label{eq:whlQ2n}
	(E_2-s^2Q_2^2)\tQ_1^2 Q_2 = (Q_2^2E_2-1)Q_2.
\end{equation}
Multiplying $(E_2-s^2Q_2^2)Q_2$ to \eqref{eq:whlQ2} from the left, we have
\begin{align}
\begin{split} \label{eq:wmQ1a}
	&(E_1-s^2Q_1^2)(E_2-s^2Q_2^2)\tQ_1^2 Q_1Q_2\\
	&\quad=(E_2-s^2Q_2^2)(Q_1^2E_1-1) Q_1Q_2.
\end{split}
\end{align}
Similarly,
\begin{align}
\begin{split} \label{eq:wnQ1a}
	&(E_1-s^2Q_1^2)(E_2-s^2Q_2^2)\tQ_1^2 Q_1Q_2\\
	&\quad=(E_1-s^2Q_1^2)(Q_2^2E_2-1)Q_1Q_2.
\end{split}
\end{align}
Subtracting \eqref{eq:wnQ1a} from \eqref{eq:wmQ1a}, we have
\begin{align*}
	\hB'_{12}(W)&(s,E_1,E_2,Q_1,Q_2)\\
	&=\{(E_2-s^2Q_2^2)(Q_1^2E_1-1)-(E_1-s^2Q_1^2)(Q_2^2E_2-1)\}Q_1Q_2
\end{align*}
as an annihilator of the summand $F(m,n;i)$.
Since
\begin{align*}
	&\{(E_2-s^2Q_2^2)(Q_1^2E_1-1)-(E_1-s^2Q_1^2)(Q_2^2E_2-1)\}Q_1Q_2\\
	&\quad=sQ_1Q_2\{(E_2-sQ_2^2)(sQ_1^2E_1-1)-(E_1-sQ_1^2)(sQ_2^2E_2-1)\},
\end{align*}
$(E_2-sQ_2^2)(sQ_1^2E_1-1)-(E_1-sQ_1^2)(sQ_2^2E_2-1)$ is also an annihilating polynomial of $F(m,n;i)$.
We put
\[
	\hB_{12}(W)(s,E_1,E_2,Q_1,Q_2)=(E_2-sQ_2^2)(sQ_1^2E_1-1)-(E_1-sQ_1^2)(sQ_2^2E_2-1).
\]
Since
\[
	\hB_{12}(W)(s,E_1,E_2,Q_1,Q_2)F(m,n;i)=0
\]
for any $i$, we have
\[
	\hB_{12}(W)(s,E_1,E_2,Q_1,Q_2)V_W(m,n)=0.
\]
\begin{rem}
As in Remark \ref{rem:VvsJ}, we see that
\[
	\{(E_2-s^2Q_2^2)(Q_1^2E_1-1)-(E_1-s^2Q_1^2)(Q_2^2E_2-1)\} \alpha(Q_1,Q_2),
\]
where $\alpha(Q_1,Q_2)=(Q_1^2-1)(Q_2^2-1)$, is an annihilating polynomial of $J_W(m,n)$.
This is a polynomial in $s^2$, $E_1$, $Q_1^2$, and $Q_2^2$ as well.
\end{rem}

\subsubsection{Generators of the ideal $\hI(W)$}
The polynomials $\Af_{1}(W)$ and $\hB_{12}(W)$ are in the left ideal $\hI(W)$ of the algebra
\[
	\A_2 = \left\{\sum _{j_1,j_2}c_{j_1,j_2}(s,Q_1,Q_2)E_1^{j_1}E_2^{j_2} \middle| c_{j_1,j_2}(s,Q_1,Q_2) \in \mathbb{Q}(s,Q_1,Q_2)\right\}.
\]
Let us `divide' $\Af_{1}(W)(s,E_1,Q_1,Q_2)$ by $\hB_{12}(W)(s,E_1,E_2,Q_1,Q_2)$ as we did for the Hopf link.
First, we give a lemma related to the property of the Ore domain.
Let $\deg(f; E_2)$ be the degree of a polynomial $f$ with respect to $E_2$.
\begin{lem} \label{lem:ore}
Let $f(s,E_2,Q_1,Q_2)$ and $g(s,E_2,Q_1,Q_2)$ be polynomials in the algebra $\A_2$.
Suppose that $\deg(g; E_2)=1$.
Then, there exist polynomials $\tilde{f}(s,E_2,Q_1,Q_2)$ and $\tilde{g}(s,E_2,Q_1,Q_2)$ such that
\[
	\tilde{g}(s,E_2,Q_1,Q_2)f(s,E_2,Q_1,Q_2) = \tilde{f}(s,E_2,Q_1,Q_2)g(s,E_2,Q_1,Q_2).
\]
Moreover, $\deg(\tilde{f}; E_2)=\deg(f; E_2)$ and $\deg(\tilde{g}; E_2)=1$.
\end{lem}
\begin{proof}
It suffices to show the case where $f$ and $g$ are monic.
We put
\[
	f(s,E_2,Q_1,Q_2) = E_2^d + \sum^{d-1}_{k=0}a_k(s,Q_1,Q_2)E_2^k,
\]
and 
\[
	g(s,E_2,Q_1,Q_2)=E_2+b(s,Q_1,Q_2),
\]
where $d=\deg(f; E_2)$, and $a_k(s,Q_1,Q_2),\ b(s,Q_1,Q_2) \in \mathbb{Q}(s,Q_1,Q_2)$.
We also put
\[
	\tilde{f}(s,E_2,Q_1,Q_2) = E_2^d + \sum^{d-1}_{k=0}A_k(s,Q_1,Q_2)E_2^k,
\]
and
\[
	\tilde{g}(s,E_2,Q_1,Q_2)=E_2+B(s,Q_1,Q_2),
\]
and find polynomials $A_k(s,Q_1,Q_2),\ B(s,Q_1,Q_2) \in \mathbb{Q}(s,Q_1,Q_2)$.
Since finding such polynomials is easy if $b(s,Q_1,Q_2)=0$, we assume that $b(s,Q_1,Q_2)\neq 0$.
Since $\tilde{g}f=\tilde{f}g$, we have
\begin{equation}\label{eq:A0}
	A_0(s,Q_1,Q_2)b(s,Q_1,Q_2)=B(s,Q_1,Q_2)a_0(s,Q_1,Q_2),
\end{equation}
\begin{align}
	\begin{split} \label{eq:Ak}
	A_{k-1}&(s,Q_1,Q_2)+A_k(s,Q_1,Q_2)b(s,Q_1,s^kQ_2)\\
	&=a_{k-1}(s,Q_1,Q_2)+a_k(s,Q_1,Q_2)B(s,Q_1,Q_2)
	\end{split}
\end{align}
for $k=1,\ldots,d-1$, and
\begin{equation} \label{eq:Ad}
	A_{d-1}(s,Q_1,Q_2)+b(s,Q_1,s^d Q_2) = a_{d-1}(s,Q_1,Q_2) + B(s,Q_1,Q_2).
\end{equation}
By the equalities \eqref{eq:A0} and \eqref{eq:Ak}, we can inductively show that
all $A_k(s,Q_1,Q_2)$, with $k=0,\ldots,d-1$, are of the form
\[
	A_k(s,Q_1,Q_2) = \xi _k(s,Q_1,Q_2) B(s,Q_1,Q_2) + \eta _k(s,Q_1,Q_2),
\]
where $\xi_k(s,Q_1,Q_2),\ \eta_k(s,Q_1,Q_2) \in \mathbb{Q}(s,Q_1,Q_2)$.
Substituting
\[
	A_{d-1}(s,Q_1,Q_2) = \xi _{d-1}(s,Q_1,Q_2) B(s,Q_1,Q_2) + \eta _{d-1}(s,Q_1,Q_2),
\]
for \eqref{eq:Ad}, we obtain $B(s,Q_1,Q_2) \in \mathbb{Q}(s,Q_1,Q_2)$ as a solution of a linear equation,
and obtain $A_k(s,Q_1,Q_2) \in \mathbb{Q}(s,Q_1,Q_2)$.
\end{proof}
\begin{prop} \label{thm:ApolydivW}
	There is a relationship between $\Af_{1}(W)$, $\Af_{2}(W)$, and $\hB_{12}(W)$ of the form
	\[
		\alpha(s,E_2,Q_1,Q_2)\Af_{1}(W) = \beta(s,E_1,E_2,Q_1,Q_2)\hB_{12}(W) + \gamma(s,Q_1,Q_2)\Af_{2}(W),
	\]
	where $\alpha(s,E_2,Q_1,Q_2),\beta(s,E_1,E_2,Q_1,Q_2)$ are non-commutative polynomials, and $ \gamma(s,Q_1,Q_2) \in \mathbb{Q}(s,Q_1,Q_2)$.
\end{prop}
\begin{proof}
Multiplying an element in $\mathbb{Q}(s,Q_1,Q_2)$, we normalize the annihilating polynomial with respect to $E_1$ as of the form
\[
	\Af'_{1}(W)(s,E_1,Q_1,Q_2)=E_1^5 + \sum^4_{k=0}d_k(s,Q_1,Q_2)E_1^k,
\]
where $d_k(s,Q_1,Q_2) \in \mathbb{Q}(s,Q_1,Q_2)$, with $k=0,\ldots,4$.
Using Lemma \ref{lem:ore}, we can decrease the degree of $\Af'_1(W)$ with respect to $E_2$ one by one.
We put
\[
	\hB_{12}(W)(s,E_1,E_2,Q_1,Q_2)=c_1(s,E_2,Q_1,Q_2)E_1+c_0(s,E_2,Q_1,Q_2).
\]
\begin{enumerate}
	\item Noting
	\[
		E_1^4c_1(s,E_2,Q_1,Q_2)E_1 = c_1(s,E_2,s^4Q_1,Q_2)E_1^5,
	\]
	we let
	\begin{equation} \label{eq:F1A}
		F_1(s,E_1,E_2,Q_1,Q_2)=c_1(s,E_2,s^4Q_1,Q_2)\Af'_{1}(W)-E_1^4 \hB_{12}(W).
	\end{equation}
	Then, $F_1$ equals
	\begin{align*}
		F_1&(s,E_1,E_2,Q_1,Q_2)= \delta _1(s,E_2,Q_1,Q_2)E_1^4  \\
		&+ \sum^3_{k=0}c_1(s,E_2,s^4Q_1,Q_2)d_k(s,Q_1,Q_2)E_1^k,
	\end{align*}
	where $\delta _1(s,E_2,Q_1,Q_2) = c_1(s,E_2,s^4Q_1,Q_2)d_4(s,Q_1,Q_2)-c_0(s,E_2,s^4Q_1,Q_2)$.
	\item By Lemma \ref{lem:ore}, there exists a polynomial
	$\tc_1(s,E_2,Q_1,Q_2)$ and $\tdel_1(s,E_2,Q_1,Q_2)$ such that
	\[
		\tc_1(s,E_2,Q_1,Q_2)\delta_1(s,E_2,Q_1,Q_2)=\tdel_1(s,E_2,Q_1,Q_2)c_1(s,E_2,s^3Q_1,Q_2).
	\]
	We let
	\begin{align}
	\begin{split}
		F_2&(s,E_1,E_2,Q_1,Q_2)\\
		&=\tc_1(s,E_2,Q_1,Q_2)F_1(s,E_1,E_2,Q_1,Q_2)-\tdel_1(s,E_2,Q_1,Q_2)E_1^3 \hB_{12}(W).
	\end{split}
	\end{align}
	Then,
	\begin{align*}
		F_2&(s,E_1,E_2,Q_1,Q_2)=\delta_2(s,E_2,Q_1,Q_2)E_1^3 \\
		&+\sum^2_{k=0}\tc_1(s,E_2,Q_1,Q_2)c_1(s,E_2,s^4Q_1,Q_2)d_k(s,Q_1,Q_2)E_1^k,
	\end{align*}
	where
	\begin{align*}
		\delta_2&(s,E_2,Q_1,Q_2)\\
		&=\tc_1(s,E_2,Q_1,Q_2)c_1(s,E_2,s^4Q_1,Q_2)d_3(s,Q_1,Q_2)\\
		&-\tdel_1(s,E_2,Q_1,Q_2)c_0(s,E_2,s^3Q_1,Q_2).
	\end{align*}
	\item By Lemma \ref{lem:ore}, there exists a polynomial
	$\tc_2(s,E_2,Q_1,Q_2)$ and $\tdel_2(s,E_2,Q_1,Q_2)$ such that
	\[
		\tc_2(s,E_2,Q_1,Q_2)\delta_2(s,E_2,Q_1,Q_2)=\tdel_2(s,E_2,Q_1,Q_2)c_1(s,E_2,s^2Q_1,Q_2).
	\]
	We let
	\begin{align}
	\begin{split}
		F_3&(s,E_1,E_2,Q_1,Q_2)\\
		&=\tc_2(s,E_2,Q_1,Q_2)F_2(s,E_1,E_2,Q_1,Q_2)-\tdel_2(s,E_2,Q_1,Q_2)E_1^2 \hB_{12}(W).
	\end{split}
	\end{align}
	Then,
	\begin{align*}
		F_3&(s,E_1,E_2,Q_1,Q_2)=\delta_3(s,E_2,Q_1,Q_2)E_1^2 \\
		&+\sum^1_{k=0}(\tc_2 \tc_1)(s,E_2,Q_1,Q_2)c_1(s,E_2,s^4Q_1,Q_2)d_k(s,Q_1,Q_2)E_1^k,
	\end{align*}
	where
	\begin{align*}
		\delta_3&(s,E_2,Q_1,Q_2)\\
		&=(\tc_2 \tc_1)(s,E_2,Q_1,Q_2)c_1(s,E_2,s^4Q_1,Q_2)d_2(s,Q_1,Q_2)\\
		&-\tdel_2(s,E_2,Q_1,Q_2)c_0(s,E_2,s^2Q_1,Q_2).
	\end{align*}
	\item By Lemma \ref{lem:ore}, there exists a polynomial
	$\tc_3(s,E_2,Q_1,Q_2)$ and $\tdel_3(s,E_2,Q_1,Q_2)$ such that
	\[
		\tc_3(s,E_2,Q_1,Q_2)\delta_3(s,E_2,Q_1,Q_2)=\tdel_3(s,E_2,Q_1,Q_2)c_1(s,E_2,sQ_1,Q_2).
	\]
	We let
	\begin{align}
	\begin{split}
		F_4&(s,E_1,E_2,Q_1,Q_2)\\
		&=\tc_3(s,E_2,Q_1,Q_2)F_3(s,E_1,E_2,Q_1,Q_2)-\tdel_3(s,E_2,Q_1,Q_2)E_1 \hB_{12}(W).
	\end{split}
	\end{align}
	Then,
	\begin{align*}
		F_4&(s,E_1,E_2,Q_1,Q_2)=\delta_4(s,E_2,Q_1,Q_2)E_1 \\
		&+(\tc_3 \tc_2 \tc_1)(s,E_2,Q_1,Q_2)c_1(s,E_2,s^4Q_1,Q_2)d_0(s,Q_1,Q_2),
	\end{align*}
	where
	\begin{align*}
		\delta_4&(s,E_2,Q_1,Q_2)\\
		&=(\tc_3 \tc_2 \tc_1)(s,E_2,Q_1,Q_2)c_1(s,E_2,s^4Q_1,Q_2)d_1(s,Q_1,Q_2)\\
		&-\tdel_3(s,E_2,Q_1,Q_2)c_0(s,E_2,sQ_1,Q_2).
	\end{align*}
	\item By Lemma \ref{lem:ore}, there exists a polynomial
	$\tc_4(s,E_2,Q_1,Q_2)$ and $\tdel_4(s,E_2,Q_1,Q_2)$ such that
	\[
		\tc_4(s,E_2,Q_1,Q_2)\delta_4(s,E_2,Q_1,Q_2)=\tdel_4(s,E_2,Q_1,Q_2)c_1(s,E_2,Q_1,Q_2).
	\]
	We let
	\begin{align}
	\begin{split} \label{eq:del5A}
		\delta_5&(s,E_2,Q_1,Q_2)\\
		&=\tc_4(s,E_2,Q_1,Q_2)F_4(s,E_1,E_2,Q_1,Q_2)-\tdel_4(s,E_2,Q_1,Q_2)\hB_{12}(W).
	\end{split}
	\end{align}
	Then,
	\begin{align*}
		\delta_5&(s,E_2,Q_1,Q_2)\\
		&=(\tc_4 \tc_3 \tc_2 \tc_1)(s,E_2,Q_1,Q_2)c_1(s,E_2,s^4Q_1,Q_2)d_0(s,Q_1,Q_2)\\
		&-\tdel_4(s,E_2,Q_1,Q_2)c_0(s,E_1,Q_1,Q_2).
	\end{align*}
\end{enumerate}
Combining \eqref{eq:F1A}-\eqref{eq:del5A}, we have
\begin{align*}
	(\tc_4 \tc_3 \tc_2 \tc_1)&(s,E_2,Q_1,Q_2)c_1(s,E_2,s^4Q_1,Q_2)\Af'_1(W)\\
	&=\{(\tc_4 \tc_3 \tc_2 \tc_1)(s,E_2,Q_1,Q_2)E_1^4 + (\tc_4 \tc_3 \tc_2 \tdel_1)(s,E_2,Q_1,Q_2)E_1^3 \\
	&+(\tc_4 \tc_3 \tdel_2)(s,E_2,Q_1,Q_2)E_1^2+(\tc_4 \tdel_3)(s,E_2,Q_1,Q_2)E_1\\
	&+\tdel_4(s,E_2,Q_1,Q_2)\}\hB_{12}(W)+\delta _5(s,E_2,Q_1,Q_2).
\end{align*}
Here, for $r$ noncommutative polynomials $p_j(s,E_2,Q_1,Q_2)$, with $j=1,\ldots,r$, we put
\[
	(p_1 \cdots p_r)(s,E_2,Q_1,Q_2) = p_1(s,E_2,Q_1,Q_2) \cdots p_r(s,E_2,Q_1,Q_2).
\]
$\delta_5$ has a degree $5$ with respect to $E_2$ and annihilates $V_W(m,n)$.
Therefore, $\delta_5$ equals $\Af_{2}(W)$ up to multiplication by an element in $\mathbb{Q}(s,Q_1,Q_2)$.
\end{proof}
\begin{rem}
Computational calculation shows that $\deg(\varepsilon_s \delta_k;E_2)=k$, with $k=1,\ldots,5$.
Specifically, $ \varepsilon _s \delta_5$ equals $(E_2+1)(E_2-Q_2^2)A_2(W)(E_2,Q_1,Q_2)$, where
\begin{align*}
	&A_2(W)(E_2,Q_1,Q_2)\\
	&\quad=Q_2^4Q_1^2E_2^3+(Q_2^4Q_1^4-Q_2^2Q_1^4+Q_2^4-2Q_2^2Q_1^2-Q_2^2+Q_1^2)E_2^2\\
	&\quad \quad+(Q_2^4Q_1^2-Q_2^2Q_1^4-2Q_2^2Q_1^2+Q_1^4-Q_2^2+1)E_2+Q_1^2.
\end{align*}
up to multiplication by an element in $\mathbb{Q}(Q_1,Q_2)$.
\end{rem}
More generally, when dividing any polynomial in the ideal $\hI(W)$ by $\hB_{12}(W)$ as polynomials in $E_1$,
the remainder is the polynomial in $s,\ E_2,\ Q_1$ and $Q_2$ since $\deg(\hB_{12}(W);E_1)=1$.
We also recall that clearing the denominators of the polynomial $\Af_1(W)$ yields the $\hA$-polynomial $\hA_1(W)$.
The above discussion implies the following theorem: 
\begin{thm} \label{thm:genhIW}
	$\hI(W) = \langle \hA_{1}(W),\ \hB_{12}(W)\rangle $ holds.
\end{thm}

\subsubsection{Evaluation at $s=1$}
The evaluation of $\Af_{1}(W)(s,E_1,Q_1,Q_2)$ at $s=1$ is
\begin{equation}
\begin{split} \label{eq:whlev}
	\varepsilon _s \Af_{1}&(W)(s,E_1,Q_1,Q_2)\\
	&=-\frac{1}{(1+Q_1^2)^2Q_1^2Q_2^2}(E_1+1)(E_1-Q_1^2)A_1(W)(E_1,Q_1^2,Q_2^2).
\end{split}
\end{equation}
On the other hand, evaluating $\hB_{12}(W)(s,E_1,E_2,Q_1,Q_2)$ at $s=1$, we have
\begin{equation} \label{eq:esAbi}
\begin{split}
	\hB_{12}(W)&(1,E_1,E_2,Q_1,Q_2)\\
	&=(Q_1^2-Q_2^2)E_1E_2+(1-Q_1^2Q_2^2)(E_1-E_2)+(Q_2^2-Q_1^2)\\
	&=B_{12}(W)(E_1,E_2,Q_1^2,Q_2^2).
\end{split}
\end{equation}
The $B$-polynomial $B_{12}(W)$ has
\[
	(E_1,E_2)=\pm(1,1),\ (Q_1^{2},Q_2^{2}),\ (Q_1^{-2},Q_2^{-2}),\ (-Q_2^{2},-Q_1^{2}),\ (-Q_2^{-2},-Q_1^{-2})
\]
as its roots. 
Evaluating the $\hA$-ideal $\hI(W)$ at $s=1$, we have $\varepsilon_s \hI(W) \subset I_W^\Phi$
from the equations \eqref{eq:whlev} and \eqref{eq:esAbi}.
This inclusion relation arises from the relationship between the partial derivatives of the potential function $\Phi $
and the differences of the summand $F$.
Recall that the polynomials $p_0(l_1,a,z)$ and $p_1(a,b,z)$ are defined by
\begin{gather}
	p_0(l_1,a,z) = \alpha z -1 - l_1(z-a)\\
	p_1(a,b,z) = (a z -1)(b z -1)(z -a)(z -b)-ab z(z^2-1)(z+1).
\end{gather}
in \ref{ssec:linkscom}.
Substituting $s=1$ and $\tE_1=1$ into the equations \eqref{eq:whlQ2} and \eqref{eq:whlE1} yields
\[
	p_0(E_1,Q_1^2,\tQ_1^2) = 0,\quad \text{and}\quad p_1(Q_1^2,Q_2^2,\tQ_1^2) = 0.
\]
Similarly, the equation
\[
	\varepsilon \left.\frac{E_2 F}{F}\right|_{s^n=Q_2,s^i=\tQ_1}=E_2
\]
is equivalent to $p_0(E_2,Q_2^2,\tQ_1^2) = 0$.
Note that a relationship between differences and partial derivatives holds for general hyperbolic links.
Expecting the factors $E_1+1$ and $E_1-Q_1^2$ in \eqref{eq:whlev} correspond to some $SL(2,\mathbb{C})$-representations
of the link group of the Whitehead link, we can formulate the following conjecture:
\begin{conj} \label{conj:linkAJ}
	For any link $\cL$, $ \varepsilon _s \hI(\cL) = I(\cL) $ holds.
\end{conj}
\appendix
\section{Another relationship between $\Af_{1}(W)$, $\Af_{2}(W)$ and $\hB_{12}(W)$} 
In this section we find another relationship between $\Af_{1}(W)$, $\Af_{2}(W)$ and $\hB_{12}(W)$.
\begin{lem} \label{lem:AWalt2}
Eliminating $\tQ_1^4$ from \eqref{eq:whlQ4} and
\[
	u_1(s,E_2,Q_2)(E_2-s^2Q_2^2)\tQ_1^4Q_2=v_1(s,E_2,Q_2)(Q_2^2E_2-1)Q_2,
\]
we have
\begin{align*}
	\hB^2_{12}(W)&(s,E_1,E_2,Q_1,Q_2)\\
	&=\{u_1(s,E_2,Q_2)(E_2-s^2Q_2^2)v_1(s,E_1,Q_1)(Q_1^2E_1-1)\\
	&\quad-u_1(s,E_1,Q_1)(E_1-s^2Q_1^2)v_1(s,E_2,Q_2)(Q_2^2E_2-1)\}Q_1Q_2.
\end{align*}
Then, $ \hB'_{12}(W)(s,E_1,E_2,Q_1,Q_2)$ is a right divisor of $\hB^2_{12}(W)(s,E_1,E_2,Q_1,Q_2)$.
\end{lem}
\begin{proof}
Recalling the factorization \eqref{eq:sfact},
we can verify
\begin{align*}
	\hB^2_{12}(W)&(s,E_1,E_2,Q_1,Q_2)\\
	&=\{u_1(s,E_2,Q_2)v_1(s,E_1,Q_1)+u_1(s,E_1,Q_1)v_1(s,E_2,Q_2)\}\\
	& \quad \times \hB'_{12}(W)(s,E_1,E_2,Q_1,Q_2)
\end{align*}
by a direct calculation.
\end{proof}
\begin{prop} \label{prop:AnnWrel}
$\hB_{12}(W)(s,E_1,E_2,Q_1,Q_2)$ is a right divisor of
\begin{align*}
	P^1_W&(s,sE_2,Q_2)Y(s,sE_2,Q_2)\Af_{1}(W)(s,E_1,Q_1,Q_2)\\
	&-P^1_W(s,sE_1,Q_1)Y(s,sE_1,Q_1)\Af_{2}(W)(s,E_2,Q_1,Q_2).
\end{align*}
See \ref{sssec:A1WL} to recall the definitions of $P^1_W(s,E_1,Q_1)$ and $Y(s,E_1,Q_1)$.
\end{prop}
\begin{proof}
It suffices to show that $\hB'_{12}(W)(s,E_1,E_2,Q_1,Q_2)$ is a right divisor of
\begin{equation} \label{eq:YP0}
	Y(s,E_2,Q_2)\tilde{P}^0_W(s,E_1,Q_1,Q_2)Q_2-Y(s,E_1,Q_1)\tilde{P}^0_W(s,E_2,Q_2,Q_1)Q_1.
\end{equation} 
Note that $\tilde{P}^0_W(s,E_1,Q_1,Q_2)$ is obtained by eliminating $\tQ_1$ from \eqref{eq:whlQ2} and
\begin{align*}
	&(s^2-s^4)\tQ_1^4+s^2(Q_1^2+Q_2^2+Q_1^{-2}+Q_2^{-2}+1)\tQ_1^2 \\
	&\quad -(Q_1^{2}Q_2^{2}+Q_1^{2}Q_2^{-2}+3+Q_1^{-2}Q_2^{2}+Q_1^{-2}Q_2^{-2})\\
	&\quad \quad +s^{-2}(Q_1^2+Q_2^2+Q_1^{-2}+Q_2^{-2}-s^{-2})\tQ_1^{-2} -s^{-4}\tQ_1^{-4}=0.
\end{align*}
The term of $\tilde{P}^0_W(s,E_1,Q_1,Q_2)Q_2$ derived from $\tQ_1^4$ is
\begin{align*}
	Y&(s,E_1,Q_1)\tQ_1^4 Q_1 Q_2\\
	&= u_2(s,E_1,Q_1)v_1(s,E_1,Q_1)(Q_1^2E_1-1)Q_1 Q_2.
\end{align*}
We put
\begin{align*}
	f&(s,E_1,E_2,Q_1,Q_2)\\
	&=Y(s,E_2,Q_2)u_2(s,E_1,Q_1)v_1(s,E_1,Q_1)(Q_1^2E_1-1)Q_1 Q_2.
\end{align*}
Then,
\begin{align*}
	&f(s,E_1,E_2,Q_1,Q_2)-f(s,E_2,E_1,Q_2,Q_1)\\
	&\quad=u_2(s,E_1,Q_1)u_2(s,E_2,Q_2)\hB^2_{12}(W)(s,E_1,E_2,Q_1,Q_2).
\end{align*}
By Lemma \ref{lem:AWalt2}, $\hB'_{12}(W)(s,E_1,E_2,Q_1,Q_2)$ is a right divisor of this term.
The term $g_{1}(s,E_1,Q_1,Q_2)$ of $\tilde{P}^0_W(s,E_1,Q_1,Q_2)Q_2$ derived from $(Q_1^2+Q_1^{-2})\tQ_1^2$ is
\begin{align*}
	g_{1}&(s,E_1,Q_1,Q_2)\\
	&=Y(s,E_1,Q_1)(Q_1^2+Q_1^{-2})\tQ_1^2 Q_1 Q_2\\
	&\quad= u_2(s,E_1,Q_1)h(s,E_1,Q_1)(Q_1^2E_1-1)Q_1 Q_2,
\end{align*}
where
\[
	h(s,E_1,Q_1)=u_1(s,E_1,Q_1)(s^2Q_1^2+s^{-2}Q_1^{-2})+1-s^{2}.
\]
On the other hand, The term $g_{2}(s,E_2,Q_1,Q_2)$ of $\tilde{P}^0_W(s,E_2,Q_2,Q_1)Q_1$ derived from $(Q_1^2+Q_1^{-2})Q_1^2$ is
\begin{align*}
	g_{2}&(s,E_2,Q_1,Q_2)\\
	&=Y(s,E_2,Q_2)(Q_1^2+Q_1^{-2})\tQ_1^2 Q_1 Q_2\\
	&\quad= u_2(s,E_2,Q_2)u_1(s,E_2,Q_2)(Q_1^2+Q_1^{-2})(Q_2^2E_2-1) Q_1 Q_2.
\end{align*}
Then,
\begin{align*}
	Y&(s,E_1,Q_1) g_{2}(s,E_2,Q_1,Q_2)\\
	&=u_2(s,E_1,Q_1)u_2(s,E_2,Q_2)u_1(s,E_2,Q_2)h(s,E_1,Q_1)\\
	&\quad \times(E_1-s^2Q_1^2)(Q_2^2E_2-1)Q_1 Q_2.
\end{align*}
Therefore,
\begin{align*}
	Y&(s,E_2,Q_2) g_{1}(s,E_1,Q_1,Q_2)-Y(s,E_1,Q_1) g_{2}(s,E_2,Q_1,Q_2)\\
	&=u_2(s,E_1,Q_1)u_2(s,E_2,Q_2)u_1(s,E_2,Q_2)h(s,E_1,Q_1)\\
	& \quad \times \hB'_{12}(W)(s,E_1,E_2,Q_1,Q_2).
\end{align*}
We can perform a similar calculation for the terms derived from $(Q_2^2+Q_2^{-2})\tQ_1^2$ and $\tQ_1^2$.
Recalling the factorization
\begin{align*}
	&u_2(s,E_1,Q_1)u_1(s,E_1,Q_1)(E_1-s^2Q_1^2)\\
	&\quad=v_2(s,E_1,Q_1)v_1(s,E_1,Q_1)(Q_1^2E_1-1),
\end{align*}
we can similarly varify that $\hB'_{12}(W)(s,E_1,E_2,Q_1,Q_2)$ is a right divisor of the terms of \eqref{eq:YP0}
derived from $\tQ_1^{-4}$ and $(Q_1^{2}+Q_2^{2}+Q_1^{-2}+Q_2^{-2}-s^2)\tQ_1^{-2}$.
Since $Y(s,E_1,Q_1)$ and $Y(s,E_2,Q_2)$ are commutative,
the term of \eqref{eq:YP0} derived from
\[
	Q_1^{2}Q_2^{2}+Q_1^{2}Q_2^{-2}+3+Q_1^{-2}Q_2^{2}+Q_1^{-2}Q_2^{-2}
\]
equals $0$.
\end{proof}

\end{document}